\documentclass[10pt,a4paper]{article}

\usepackage[english]{babel}
\usepackage[T1]{fontenc}
\usepackage{bbm}
\usepackage{graphicx}
\usepackage{verbatim}
\usepackage{amsmath}
\usepackage{amssymb}
\usepackage{amsthm}
  \theoremstyle{plain}
  \newtheorem{theorem}{Theorem}
  \newtheorem{lemma}{Lemma}
  
  \newtheorem{corollary}{Corollary}
  \newtheorem{assumption}{Assumption}
  \newtheorem{definition}{Definition}
  
  \theoremstyle{remark}
  \newtheorem*{remark}{Remark}
\usepackage{mathrsfs}
\usepackage{mathtools}
\usepackage[left=2.7cm,right=2.7cm,top=3.1cm,bottom=3.1cm]{geometry}
\usepackage[colorlinks=true, allcolors=blue, hypertexnames=false]{hyperref}

\usepackage[round]{natbib}
\usepackage[ruled,linesnumbered]{algorithm2e}
  \SetKwInput{KwArg}{Argument}

 \allowdisplaybreaks[4]

\begin{document}

\title{Decentralized Stochastic Linear-Quadratic Optimal Control with Risk Constraint and Partial Observation}
\author{Hui Jia \and Yuan-Hua Ni}
\date{\today}

\maketitle

\begin{abstract}
This paper addresses a risk-constrained decentralized stochastic linear-quadratic optimal control problem with one remote controller and one local controller, where the risk constraint is posed on the cumulative state weighted variance in order to  reduce the oscillation of system trajectory. 
In this model, local controller can only partially observe the system state, and sends the estimate of state to remote controller through an unreliable channel, 
whereas the channel from remote controller to  local controllers is perfect. 
For the considered constrained optimization problem, we first  punish the risk constraint into cost function through Lagrange multiplier method, and the resulting augmented cost function will include a quadratic mean-field term of state. 
In the sequel, for any but fixed multiplier, explicit solutions to finite-horizon and infinite-horizon mean-field decentralized  linear-quadratic problems are derived together with necessary and sufficient condition on the mean-square stability of optimal system.
Then, approach to find the optimal Lagrange multiplier is presented based on bisection method. Finally, two numerical examples are given to show the efficiency of the obtained results.

\end{abstract}

\section{Introduction}\label{sec:introduction}

{Achieving good average performance is often the goal of optimal control, especially in modern networked control systems (NCSs), such as the automated highway systems (\cite{Horowitz00}), unmanned aerial vehiclesn (\cite{Chen22}), electronic systems (\cite{Gee10}) and manufacturing systems (\cite{Ding-yi}).
However, unlikely, atypical or unexpected events may lead to catastrophic consequences; for example, unmanned aerial vehicle deviates too much from the given trajectory in hostile environments, or autonomous vehicle hits a wall or a pedestrian. In this case, only optimizing the total expected
cost may not be enough, and risk constraints should be included
in the optimization process as objectives or restrictions.}


Various risk measures are reported in existing literature, such as the risk-sensitive criterion (\cite{Jacobson73,Whittle90}) and conditional value-at-risk (CVaR) (\cite{RTyrrell00}).
	%
Risk-sensitive control, in that the quadratic cost function of  standard linear-quadratic-Gaussian (LQG) treatment is replaced by the exponential of a quadratic, gives the so-called linear exponential quadratic Gaussian (LEQG) formulation. 
If the noise covariance is large enough, the optimal controller of LEQG will no longer exist; this differs from the LQG setting.   
%
%
CVaR is a risk measure in optimizing or hedging financial instrument portfolios that quantifies the amount of tail risk an investment portfolio has. 
CVaR is derived by taking a weighted average of the ``extreme" losses in the tail of distribution of possible returns, beyond the value at risk cutoff point.


%
Unlike the above mentioned ones, \cite{tsiamis21} proposes a new risk measure for classical linear-quadratic (LQ) problem,  which deals with a LQ problem with general noises and partial observation.
In \cite{tsiamis21}, the adopted risk measure is state's  cumulative expected predictive (conditional) variance, and the constraint is  posed by letting this risk measure be smaller than a given level; this is to restrain the occurrence of the phenomenon that system state can grow arbitrarily large under less probable, yet extreme events. 
%
Note that the constraint is on state's fourth-order moment with some conditional expectations; by exploring some particular structure, the risk measure reduces to a quadratic function of state's estimation that is parameterized by some higher-order moments of prediction error. Then, the considered risk-constrained LQ  problem is equivalently expressed as a sequential
variational quadratically constrained quadratic programming (QCQP). By using the Lagrangian duality theory, the explicit expression of optimal risk-aware controller is obtained for an arbitrary but fixed Lagrange multiplier. Finally, an optimal Lagrange multiplier may be efficiently discovered via trivial bisection.

In this paper, we will consider a risk-constrained optimal control problem in NCSs. In fact, the study of networked optimal control has attracted much attention from system control community. 
In particular, \cite{Liang18} studies a networked optimal control problem with a local controller and a remote controller; local controller can perfectly observe the system state, and transmits the obtained state information to remote controller through an unreliable channel.
Then, \cite{Xiao20} generalizes the model of \cite{Liang18} to the case that local controller just accesses partially to the system state and that the exact values of state are still transmitted to remote controller yet. 
Simultaneously, \cite{Seyed19} considers a networked optimal control problem with $N$ local controllers and a remote controller, and same to \cite{Liang18} the system state is assumed to be available perfectly to each local controller (\cite{Seyed19}). 
Noting that the works here fail to care about the risk, \cite{tsiamis21} calls such problems the risk-neutral.

Namely, the mentioned works on networked optimal control are  generally minimizing the expectation of total cost. Yet mathematical expectation just reflects the average performance, and extreme cases  may still occur on the premise of small probability. Actually, in order to avoid reaching the extreme situation, it is natural to hope that the optimal state process will not be too sensitive to possible changes and one way to achieve this is trying to keep the variation of system state small (\cite{Yong13}).
Hence, the cumulative state weighted variance might be a proper risk measure to limit the statistical variability of system state, which is the concern of this paper.

This paper addresses a risk-constrained decentralized stochastic LQ problem with partial observation. The contributions and novelties are stated in what follows.


\begin{itemize}
	\item [(1)] The cumulative state weighted variance is adopted as a risk measure in this paper. Whereas the state's cumulative expected predictive (conditional) variance, some  fourth-order moment involving conditional expectation, is adopted by \cite{tsiamis21} that will reduce to a quadratic function of state's estimation parameterized by some higher-order moments of prediction error; and the augmented  cost function of unconstrained LQ problem will be similarly transformed.  
	
	By applying Lagrangian duality theory, the unconstrained (risk-aware) cost function of this paper will include some  mean-field term of state, and the resulting unconstrained  optimal control problem becomes a mean-field LQ problem with partial observation. Then, orthogonal decomposition and maximum principle techniques are employed to derive the risk-aware optimal controls for finite-horizon problem and infinite-horizon problem, respectively.  
	In contrast, \cite{tsiamis21} obtains her results by standard method of dynamic programming.

	Moreover, the model of \cite{tsiamis21} is centralized, whereas this paper handles risk-constrained decentralized optimal control.

	\item[(2)] Compared with existing mean-field stochastic optimal controls (\cite{Yong13}, \cite{Ni15} and \cite{Zhang16}) with perfect state observation, this paper handles the decentralized mean-field optimal control with partial observation.

	\item[(3)] For the finite-horizon problem with fixed Lagrange multiplier, necessary and sufficient condition on the existence of risk-aware optimal control are presented together with the closed-form expression. 
	
	For the infinite-horizon problem with fixed Lagrange multiplier, necessary condition on the stabilization in the mean-square sense is presented firstly for the controlled system  without additive noises. Then, necessary and sufficient condition on the boundedness of optimal state in the mean-square sense is characterized for the controlled  system with  additive noises.
	
	In addition,  using the bisection method (similarly to \cite{tsiamis21}), the optimal Lagrange multiplier is calculated. 
\end{itemize}	
%

The rest of this paper is organized as follows. Section 2 presents the considered remote-local decentralized optimal control problem with partial observation and risk constraint. In Section 3, the constrained optimization problem of Section 2 is transformed into a unconstrained one by using Lagrange duality theory. Section 4 gives the explicit expression of optimal control and study the stability problem for any fixed Lagrange multiplier. Section 5 gives the method for searching the optimal Lagrange multiplier. Two numerical examples are presented in Section 6, and Section 7 concludes this paper. 

\textit{\textbf{Notations:}}
Let  $\mathbb{P}(\cdot)$ be the probability measure and $\mathbb{E}(\cdot)$ be the mathematical expectation. For $a\leq b$, the collection of vectors $x_{a:b}$ is a short hand for $(x_a,x_{a+1},\cdots,x_b)$. $\mathbb{R}_{+}$ denotes the set of positive real numbers. $A \geq 0$ $(>0)$ denotes that $A$ is a positive semi-definite (positive definite) matrix. The transpose operation and inverse operation are denoted by $(\cdot)^{\prime}$ and $(\cdot)^{-1}$, respectively. 
$\mbox{tr}(A)$ denotes the trace of matrix A. Denotes diag($A$,$B$) as the diagonal block of matrices $A$, $B$. $I$ represents the unit matrix with appropriate dimensions, and $\sigma(x)$ denotes the $\sigma$-algebra generated by random variable $x$.

\section{Problem Formulation}

Consider a discrete-time plant, shown in Fig 1 below, with a local controller and a remote controller, which evolves according to some discrete-time stochastic difference equation
\begin{equation}\label{system}
	\begin{split}
		x_{k+1}=Ax_k+B^Lu^L_k+B^Ru^R_k+w_k.
	\end{split}
\end{equation}
Here, $x_k\in \mathbb{R}^n$, $u_k^L \in \mathbb{R}^{m_1}$, and $u_k^R \in \mathbb{R}^{m_2}$ are the state, local controller  and remote controller, respectively. 
The initial state $x_0$
is Gaussian with mean $\bar{x}_0$ and covariance $\Sigma_{init}$.
 $\{\omega_k\}$ is a sequence of independent and identically distributed (i.i.d)  Gaussian random variables with mean zero and covariance $Q_\omega$. At any time $k$, local controller makes a noisy observation and sends the state's estimator to the remote controller through an unreliable channel. Hence, the observation models are as follows:
\begin{align}\label{obervation-1}
	y_{k}^L&=Cx_k+v_k,\\
	y_{k}^R&=\eta_{k}\hat{x}_{k\mid k}^L, 
\end{align}
where $y_k^L$, $y^R_k\in \mathbb{R}^n$ are the observations of  local and remote controller with $\hat{x}_{k\mid k}^L$ the optimal estimator of state $x_k$ of local controller that is defined below. $\{v_k\}$ is a sequence of i.i.d. Gaussian random variables with mean zero and covariance  $Q_v$. $\eta_k$ is a Bernoulli random variable that describes the unreliable channel from local controller to remote controller.

\begin{figure}[!htb]
	\centering
	\includegraphics[width=0.6\hsize]{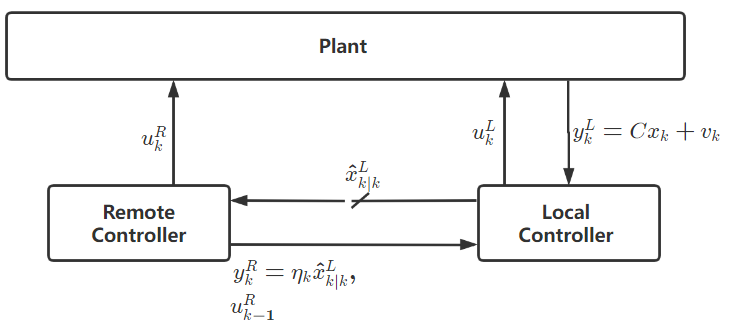}
	\caption{System model}
	\label{fig1}
\end{figure}

There are two types of communication channels in the model, namely, the unreliable uplink from local controller to remote controller, and the perfect downlink from remote controller to local controller. 
Through the unreliable channel, at time $k$ the local controller chooses to send $\hat{x}_{k\mid k}^L$ to the remote controller with channel failure probability $p$.
$\eta_k$ above describes the unreliable channel from local controller to remote controller, namely, $\eta_k=1$ indicates the successful transmission and the remote controller receives $\hat{x}_{k\mid k}^L$, and $\eta_k=0$ means the channel transmission failure and data loss. 
The channel from remote controller to local controller is perfect. Therefore, remote controller shares $y_k^R$ and $u_{k-1}^R$ to local controller at each time $k$. 

Let $\mathcal{F}^L_k=\sigma(y^L_0,\cdots,y^L_k,y^R_0,\cdots,y^R_k)$ and $\mathcal{F}^R_k=\sigma(y^R_0,\cdots,y^R_k)$. Based on this notation, we introduce the following admissible control set of $u=(u_{0:N}^L,u_{0:N}^R)$:
\begin{equation}
	\mathcal{U}_{ad}=\Big\{u\;\Big|\; u_k^L \; \mbox{is}\; \mathcal{F}^L_k\mbox{-measurable}, u_k^R \;\mbox{is}\; \mathcal{F}^R_k\mbox{-measurable},~~k=0,\cdots,N \Big\}. 	
\end{equation}
The cost function that is to be minimized is given by
\begin{align}\label{cost-1}
	J(u)=\mathbb{E} \Bigg\{\sum_{k=0}^{N}\Big{[}x_{k}^{\prime}Qx_k+(u_k^{L})^\prime R^Lu_k^L+(u_k^{R})^\prime R^Ru_k^R\Big{]}+x_{N+1}^{\prime}Gx_{N+1}\Bigg\},
\end{align}
where $Q$, $R^L$, $R^R$ and $G$ are positive semi-definite matrices. 

The decentralized LQG problem (\ref{system})-(\ref{cost-1}) is risk-neutral, since it optimizes the performance only on average. Still, even if the average performance is good, the state can grow arbitrarily large under less probable yet extreme events, when the variance of system noises is large. In other words, the state may exhibit large variability. To deal with this issue, we add a risk constraint on the state posed as
\begin{align}\label{constraint-0}
	J_R(u)=\mathbb{E}\Bigg\{\sum_{k=0}^{N+1}\big(x_k-\mathbb{E}x_k\big)^{\prime}Q\big(x_k-\mathbb{E}x_k\big)\Bigg\} \leq {\epsilon}.
\end{align}
Here, the adopted constraint  is the cumulative state weighted variance. By simply decreasing $\epsilon$, we increase the risk-awareness. Hence, our risk-constrained problem not only forces on decreasing the cost (\ref{cost-1}), but also explicitly restricts the variability of state. Therefore, the considered optimization problem is formulated as follows, which offers a way to trade-off between average performance and risk.

\textbf{Problem (CLQ).} Solve the optimization problem 
\begin{eqnarray}
	\begin{array}{cl}
		\min\limits_{u\in \mathcal{U}_{ad}}&J(u)=\mathbb{E} \Bigg\{\sum\limits_{k=0}^{N}\Big{[}x_{k}^{\prime}Qx_k+(u_k^{L})^\prime R^Lu_k^L+(u_k^{R})^\prime R^Ru_k^R\Big{]}+x_{N+1}^{\prime}Gx_{N+1}\Bigg\}\\[1mm]
		s.t.&J_R(u)=\mathbb{E}\Bigg\{\sum\limits_{k=0}^{N+1}\big(x_k-\mathbb{E}x_k\big)^{\prime}Q\big(x_k-\mathbb{E}x_k\big)\Bigg\} \leq {\epsilon},\\[1mm]
		&x_{k+1}=Ax_k+B^Lu_k^L+B^Ru_k^R+\omega_k, u\in \mathcal{U}_{ad},
	\end{array}
\end{eqnarray}
and the minimizer $u^*\in \mathcal{U}_{ad}$ is called an optimal risk-constrained control of Problem (CLQ). 


\section{Lagrangian Duality}

Define the Lagrange dual function of Problem (CLQ)
\begin{equation}
	\mathcal{L}(u,\mu)\triangleq J(u)+\mu(J_R(u)-{\epsilon}),
\end{equation}
where $\mathcal{L}: \mathcal{U}_{ad}\times \mathbb{R}_{+}\rightarrow \mathbb{R}$ and $\mu\geq 0$ is the Lagrangian multiplier. Define the dual function $D : \mathbb{R}_{+}\rightarrow \mathbb{R}$
\begin{equation}\label{dual}
	D(\mu)\triangleq \mathop{\inf}\limits_{u \in \mathcal{U}_{ad}} \,\mathcal{L}(u,\mu).
\end{equation}
%
%
%
Then, the dual problem of Problem (CLQ) is formulated as
\begin{equation}\label{dual-problem}
	\mathop{\sup}\limits_{\mu \geq 0}\,D(\mu)=\mathop{\sup}\limits_{\mu \geq 0} \mathop{\inf}\limits_{u \in \mathcal{U}_{ad}} \,\mathcal{L}(u,\mu).
\end{equation}


The following result states the optimal condition of Problem  (CLQ) according to the Lagrangian duality theory (Theorem 1 of  \cite{Tsiamis20}, and Theorem 4.10 of \cite{Andrzej06}).
\newtheorem{thm}{\bf Theorem}
\begin{theorem}\label{Optimality Conditions}
	Suppose that there exists a feasible control-multiplier pair $(u^{\ast},\mu^{\ast})\in \mathcal{U}_{ad} \times \mathbb{R}_+$ such that
	
	1) $\mathcal{L}(u^{\ast}(\mu^{\ast}),\mu^{\ast})=\min_{u \in \mathcal{U}_{ad}}\mathcal{L}(u,\mu^{\ast})=D(\mu^{\ast})$;
	
	2) $J_R(u^{\ast})\leq{\epsilon}$, i.e., {the dual risk} constraint of Problem (CLQ) is satisfied by control policy $u^{\ast}$;
	
	3) $\mu^{\ast}(J_R(u^{\ast})-{\epsilon})=0$, i.e., the complementary slackness holds.\\
	Then, $u^{\ast}$ is optimal for Problem (CLQ) and $\mu^{\ast}$ is optimal for the dual problem ($\ref{dual-problem}$), and further there exhibits zero duality gap, that is, $D^{\ast}=J^{\ast}$.
\end{theorem}
\section{Optimal Risk-Constrained Control}

\subsection{Finite-Horizon case}

Let $\mu\geq0$ be arbitrary but fixed. The Lagrangian function $\mathcal{L}$ is expressed as
\begin{align}
	\mathcal{L}(u,\mu)=&\mathbb{E}\Bigg\{\sum_{k=0}^{N}\Big{[}x_k^{\prime}Q_{\mu}x_k-2\mu x_k^{\prime}Q\mathbb{E}x_k+\mu \mathbb{E}x_k^{\prime}Q\mathbb{E}x_k+(u_k^{L})^\prime  R^Lu_k^L+(u_k^{R})^\prime R^Ru_k^R\Big{]}\nonumber\\ 
	&+x_{N+1}^{\prime}G_{\mu}x_{N+1}-2\mu x_{N+1}^{\prime}Q\mathbb{E}x_{N+1}+\mu \mathbb{E}x_{N+1}^{\prime}Q\mathbb{E}x_{N+1}\Bigg\}+g(\mu),
\end{align}
where
\begin{align*}
	Q_\mu=Q+\mu Q,\quad
	G_\mu=G+\mu Q,\quad
	g(\mu)=-\mu \epsilon.
\end{align*}
To this end, define
\begin{align}\label{bar_J}
	\bar{J}(\mu)=& \mathbb{E}\Bigg\{\sum\limits_{k=0}^{N}\Big{[}x_k^{\prime}Q_{\mu}x_k-\mu \mathbb{E}x_k^{\prime}Q\mathbb{E}x_k+(u_k^{L})^\prime R^Lu_k^L+(u_k^{R})^\prime R^Ru_k^R\Big{]} \nonumber\\
	&\quad+x_{N+1}^{\prime}G_{\mu}x_{N+1}-\mu \mathbb{E}x_{N+1}^{\prime}Q\mathbb{E}x_{N+1}\Bigg\},
\end{align}
which implies
\begin{align}
	D(\mu)=\inf_{u\in \mathcal{U}_{ad}} \mathcal{L}(u,\mu)=\bar{J}^{*}(\mu)+g(\mu), \label{13}
\end{align}
with $\bar{J}^{*}(\mu)=\min\limits_{u\in \mathcal{U}_{ad}}\bar{J}(\mu)$.

\textbf{Problem (FLQ).} For fixed multiplier $\mu\geq 0$, find an optimal control $(u_k^{R*}, u_k^{L*})$ that minimizes the function $\bar{J}(\mu)$, i.e.,
\begin{align}\label{Optimal-control-2}
	\bar{J}(u_k^{R*}, u_k^{L*})=\min\limits_{u\in \mathcal{U}_{ad}} \bar{J}(\mu)=&\min\limits_{u\in \mathcal{U}_{ad}} \mathbb{E}\Bigg\{\sum\limits_{k=0}^{N}\Big{[}x_k^{\prime}Q_{\mu}x_k-\mu \mathbb{E}x_k^{\prime}Q\mathbb{E}x_k+(u_k^{L})^\prime R^Lu_k^L+(u_k^{R})^\prime R^Ru_k^R\Big{]}\nonumber\\
	&\quad~~~~\,+x_{N+1}^{\prime}G_{\mu}x_{N+1}-\mu \mathbb{E}x_{N+1}^{\prime}Q\mathbb{E}x_{N+1}\Bigg\}.
\end{align}

Let 
\begin{align*}
	\hat{u}_k^L=\mathbb{E}[u_k^L\mid \mathcal{F}_k^R],\;
	\tilde{u}_k^L=u_k^L-\hat{u}_k^L.
\end{align*}
Obviously, the following properties can be readily obtained:
\begin{align*}
	\mathbb{E}[\tilde{u}_k^L\mid\mathcal{F}_k^L]=\tilde{u}_k^L,\;
	\mathbb{E}[\hat{u}_k^L\mid\mathcal{F}_k^L]=\hat{u}_k^L,\;
	\mathbb{E}[\tilde{u}_k^L\mid\mathcal{F}_k^R]=0.
\end{align*} 
Then, we rewrite $(\ref{system})$ and $(\ref{bar_J})$ as
\begin{align}
	x_{k+1}&=Ax_k+BU_k+B^L\tilde{u}_k^L+\omega_k,\label{new-system-form}\\\nonumber
	\bar{J}(\mu)&=\mathbb{E}\Bigg\{\sum\limits_{k=0}^{N}\Big{[}x_k^{\prime}Q_{\mu}x_k-\mu \mathbb{E}x_k^{\prime}Q\mathbb{E}x_k+U_k^\prime RU_k+(\tilde{u}_k^L)^\prime R^L\tilde{u}_k^L\Big{]}\\ \nonumber
	&\quad+x_{N+1}^{\prime}G_{\mu}x_{N+1}-\mu \mathbb{E}x_{N+1}^{\prime}Q\mathbb{E}x_{N+1}\Bigg\}
\end{align}
where $U_k=\begin{bmatrix}
	\hat{u}_k^L\\u_k^R
\end{bmatrix}$, $B=\begin{bmatrix}
	B^L\;B^R
\end{bmatrix}$ and $R=\begin{bmatrix}
	R^L&0\\0&R^R
\end{bmatrix}$.

We may now derive the solution to $(\ref{dual})$, which is one of main results of this paper  and provides optimal local and remote control for every fixed multiplier $\mu\geq 0$.
Before showing the optimal strategies for  fixed $\mu$, we first provide the optimal state estimators of  the two controllers.

\begin{lemma}\label{es}
	The optimal state estimators of local and remote controllers are given by
	\begin{align}
		&\hat{x}^{L}_{k \mid k}=\mathbb{E}[x_k\mid \mathcal{F}_k^L]=\hat{x}^L_{k \mid k-1}+W_{k}(y_k^L-C\hat{x}^L_{k \mid k-1}),\label{L_es}\\
		&\hat{x}^L_{k \mid k-1}=\mathbb{E}[x_k\mid \mathcal{F}_{k-1}^L]=A\hat{x}^L_{k-1 \mid k-1}+BU_{k-1}+B^L\tilde{u}_{k-1}^L,\\
		&\hat{x}^R_{k \mid k}=\mathbb{E}[x_k\mid \mathcal{F}_k^R]=\eta_{k}\hat{x}^L_{k \mid k}+(1-\eta_{k})\hat{x}^R_{k \mid k-1},\label{es_x_R}\\
		&\hat{x}^R_{k \mid k-1}=\mathbb{E}[x_k\mid \mathcal{F}_{k-1}^R]=A\hat{x}^R_{k-1 \mid k-1}+BU_{k-1};
	\end{align}
here, $W_k=\Sigma^L_{k\mid k-1}C^{\prime}(C\Sigma^L_{k\mid k-1}C^{\prime}+Q_v)^{-1}$ and $\Sigma_{k\mid k-1}^L$ is the estimation error covariance that satisfies 
\begin{align*}
	&\Sigma_{k\mid k-1}^L=A\Sigma_{k-1\mid k-1}^LA^{\prime}+Q_\omega,\\
	&\Sigma_{k\mid k}^L=(I-W_kC)\Sigma_{k\mid k-1}^L(I-W_kC)^{\prime}+W_kQ_vW_k^{\prime}
\end{align*}
with initial value $\hat{x}^L_{0\mid -1}=\hat{x}^R_{0\mid -1}=\bar{x}_0$ and $\Sigma_{0\mid -1}^L=\Sigma_{init}$.
\begin{proof}
	The optimal estimator $\hat{x}_{k\mid k}^L$ can be obtained by using the standard Kalman filtering. It remains to show how to calculate the optimal estimator $\hat{x}_{k\mid k}^R$. If $\eta_k=1$,  we have $\hat{x}_{k\mid k}^R=\mathbb{E}(x_k\mid \mathcal{F}_k^R)=\mathbb{E}(\hat{x}_{k\mid k}^L\mid \mathcal{F}_k^R)=\hat{x}^L_{k \mid k}$, else $\hat{x}_{k\mid k}^R=\hat{x}_{k\mid k-1}^R$.
\end{proof}
\end{lemma}

%

To this end, we define the following three Riccati equations:
\begin{align}
	Z_k&=A^{\prime}Z_{k+1}A+Q_\mu-K_k^{\prime}\Upsilon_kK_k,\label{Z}\\
	X_k&=A^{\prime}\Theta_{k+1}A+Q_\mu-L_k^{\prime}\Lambda_k^{-1}L_k,\label{X}\\
	S_k&=A^{\prime}S_{k+1}A-\mu Q+K_k^{\prime}\Upsilon_kK_k-N_k^{\prime}M_k^{-1}N_k,\label{S}
\end{align}
where
\begin{equation}\label{DRE}
	\left\{
	\begin{aligned}
		\Upsilon_k&=B^{\prime}Z_{k+1}B+R,\\
		K_k&=\Upsilon_k^{-1}B^{\prime}Z_{k+1}A,\\
		\bar{K}_k&=M_k^{-1}N_k-\Upsilon_k^{-1}B^{\prime}Z_{k+1}A, \\
		M_k&=B^{\prime}Z_{k+1}B+B^{\prime}S_{k+1}B+R,\\
		N_k&=B^{\prime}Z_{k+1}A+B^{\prime}S_{k+1}A,\\ 
		\Lambda_k&=B^{L^{\prime}}\Theta_{k+1}B^L+R^L,\\
		L_k&=B^{L^{\prime}}\Theta_{k+1}A,\\
		\Theta_k&=(1-p)Z_k+pX_k
	\end{aligned}
	\right.
\end{equation}
with terminal values $Z_{N+1}=X_{N+1}=G_\mu,S_{N+1}=-\mu Q$.

Applying Pontryagin's maximum principle to  system $(\ref{system})$ and cost function $(\ref{bar_J})$, we have the following {costate equations:}
\begin{align}
	\lambda_{k-1}^*&=\mathbb{E}[A^{\prime}\lambda_k^*+Q_{\mu}x_k^*-\mu Q \mathbb{E}x_k^*\mid \mathcal{F}_k^L],\label{init_ladbd}\\
	0&=\mathbb{E}[(B^{L})^\prime\lambda_k^*+R^Lu_k^{L*}\mid \mathcal{F}_k^L],\\
	0&=\mathbb{E}[(B^{R})^\prime\lambda_k^*+R^Ru_k^{R*}\mid \mathcal{F}_k^R],\\
	\lambda_{N}^*&=\mathbb{E}[G_\mu x_{N+1}^*-\mu Q\mathbb{E}x_{N+1}^*\mid \mathcal{F}_{N+1}^L];\label{init_4}	 
\end{align}
here, $\{x_k^*\}$ is the optimal state that corresponds to the optimal controller $(u^{L*}_k, u^{R*}_k)$.

\begin{lemma}\label{lemma2}
	The costate equations (\ref{init_ladbd})-(\ref{init_4}) can be rewritten as
	\begin{align}
		\lambda_{k-1}^*&=\mathbb{E}[A^{\prime}\lambda_k^*+Q_{\mu}x_k^*-\mu Q \mathbb{E}x_k^*\mid \mathcal{F}_k^L],\label{co46}\\
	0&=\mathbb{E}[B^{\prime}\lambda_k^*\mid \mathcal{F}_k^R]+RU_k^*,\label{new-Uk}\\
	0&=\mathbb{E}[(B^L)^{\prime}\lambda_k^*\mid \mathcal{F}_k^L]-\mathbb{E}[(B^L)^{\prime}\lambda_k^*\mid \mathcal{F}_k^R]+R^L\tilde{u}_k^{L*},\label{new-tiu}\\
		\lambda_{N}^*&=\mathbb{E}[G_\mu x_{N+1}^*-\mu Q\mathbb{E}x_{N+1}^*\mid \mathcal{F}_{N+1}^L]\label{lamde_n1}.
	\end{align}
\begin{proof}
 The proof is similar to that
	of \cite{Liang18}. Thus we omit here.
\end{proof}
\end{lemma}

\begin{lemma}\label{lemma3}
	Let $M_k>0$, $\Upsilon_k>0$ and $\Lambda_k>0$ for $k=0,\cdots,N+1$. Then, 
	\begin{align}\label{lamde_form}
		\lambda_{k-1}^*=Z_k\hat{x}^{R*}_{k\mid k}+X_{k}(\hat{x}^{L*}_{k\mid k}-\hat{x}^{R*}_{k\mid k})+S_k\mathbb{E}\hat{x}^{R*}_{k\mid k},
	\end{align}
where $\hat{x}_{k\mid k}^{L*}$ and $\hat{x}_{k\mid k}^{R*}$ are the estimation of optimal state $x_k^*$ of local controller and remote controller, respectively.
%

\begin{proof}
	We will show by induction that $\lambda_{k-1}^*$ has the form $(\ref{lamde_form})$ for $k=N+1,\cdots,0$.
	Firstly, noting $(\ref{lamde_n1})$, $Z_{N+1}=X_{N+1}=G_{\mu}$ and $S_{N+1}=-\mu Q$, it is obvious that $(\ref{lamde_form})$ holds for $k=N+1$.
	For $k=N$, by making use of $(\ref{new-system-form})$ and $(\ref{lamde_n1})$, the equality $(\ref{new-Uk})$ can be written as
	\begin{align}\label{24}
		0&=\mathbb{E}[B^{\prime}\lambda_N^*\mid \mathcal{F}_N^R]+RU_N^*\nonumber\\
		&=B^{\prime}\mathbb{E}[G_\mu(Ax_N^*+BU_N^*+B^L\tilde{u}_k^{L*}+\omega_N)\mid \mathcal{F}_N^R]-\mu B^{\prime}Q \mathbb{E}[Ax_N^*+BU_N^*+B^L\tilde{u}_k^{L*}+\omega_N]+RU_N^* \nonumber\\
		&=B^{\prime}G_{\mu}A\hat{x}_{N\mid N}^{R*}+B^{\prime}G_\mu BU_N^*-\mu B^{\prime}QA\mathbb{E}\hat{x}_{N\mid N}^{R*}-\mu B^{\prime}QB\mathbb{E}U_N^*+RU_N^*.
	\end{align}
	Taking mathematical expectation on both sides of $(\ref{24})$, it yields that
	\begin{align*}
		(B^{\prime}G_\mu B+R-\mu B^{\prime}QB)\mathbb{E}U_N^*=-(B^{\prime}G_\mu A-\mu B^{\prime}QA)\mathbb{E}\hat{x}_{N\mid N}^{R*}.
	\end{align*}
	Hence,
	\begin{align}\label{26}
		\mathbb{E}U_N^*=-M_N^{-1}N_N \mathbb{E}\hat{x}_{N\mid N}^{R*}.
	\end{align}
	Taking $(\ref{26})$ into $(\ref{24})$, we get
	\begin{align*}
		(B^{\prime}G_\mu B+R)U_N^*=-B^{\prime}G_\mu A\hat{x}^{R*}_{N\mid N}-(\Upsilon_N M_N^{-1}N_N-B^{\prime}G_\mu A)\mathbb{E}\hat{x}_{N\mid N}^{R*}.
	\end{align*}
	Thus, the optimal $U_N^*$ is
	\begin{align}
		U_N^*&=-\Upsilon_N^{-1}B^{\prime}G_\mu A\hat{x}^{R*}_{N\mid N}-(M_N^{-1}N_N-\Upsilon_N^{-1}B^{\prime}G_\mu A)\mathbb{E}\hat{x}_{N\mid N}^{R*}\nonumber\\
		&=-K_N\hat{x}^{R*}_{N\mid N}-\bar{K}_N\mathbb{E}\hat{x}^{R*}_{N\mid N}\label{28}.
	\end{align}
	By making using of $(\ref{new-system-form})$, $(\ref{lamde_n1})$ and $(\ref{28})$, the equality  $(\ref{new-tiu})$ becomes
	\begin{align*}
		0&=(B^L)^{\prime}G_{\mu}\mathbb{E}[Ax_N^*+BU_N^*+B^L\tilde{u}_N^{L*}+w_N\mid \mathcal{F}_N^L]\\
		&\quad-(B^L)^{\prime}G_{\mu}\mathbb{E}[Ax_N^*+BU_N^*+B^{L}\tilde{u}_k^{L*}+w_N\mid \mathcal{F}_N^R]+R^L\tilde{u}_N^{L*} \nonumber\\
		&=(B^L)^{\prime}G_{\mu}A(\hat{x}_{N\mid N}^{L*}-\hat{x}_{N\mid N}^{R*})+(B^L)^{\prime}G_\mu B^L\tilde{u}_N^{L*}+R^L\tilde{u}_N^{L*}.
	\end{align*}
	Thus, the optimal $\tilde{u}_N^{L*}$ is 
	\begin{align}
		\tilde{u}_N^{L*}&=-[(B^L)^{\prime}G_\mu B^L+R^L]^{-1}(B^L)^{\prime}G_\mu A(\hat{x}_{N\mid N}^{L*}-\hat{x}_{N\mid N}^{R*})\nonumber\\
		&=-\Lambda_N^{-1}L_N(\hat{x}_{N\mid N}^{L*}-\hat{x}_{N\mid N}^{R*})\label{30}.
	\end{align}
	By applying $(\ref{new-system-form})$, $(\ref{lamde_form})$, $(\ref{26})$, $(\ref{28})$ and $(\ref{30})$, it follows from $(\ref{lamde_form})$ that
	\begin{align*}
		\lambda_{N-1}^*&=(Q_\mu+A^{\prime}G_\mu A-A^{\prime}G_\mu B^{L^{\prime}}\Lambda_N^{-1}L_N)(\hat{x}^{L*}_{N\mid N}-\hat{x}_{N\mid N}^{R*})+(Q_\mu+A^{\prime}G_\mu A-A^{\prime}G_\mu BK_N)\hat{x}_{N\mid N}^{R*} \nonumber\\
		&\quad +(K_N^{\prime}\Upsilon_N^{-1}K_N-N_NM_N^{-1}N_N-\mu A^{\prime}QA-\mu Q)\mathbb{E}\hat{x}_{N\mid N}^{R*} \nonumber\\
		&=Z_N\hat{x}_{N\mid N}^{R*}+X_N(\hat{x}_{N\mid N}^{L*}-\hat{x}_{N\mid N}^{R*})+S_{N}\mathbb{E}\hat{x}_{N\mid N}^{R*},
	\end{align*}
	which implies that $(\ref{lamde_form})$ holds for $k=N$.

	To complete the induction, we take any $n$ with $0\leq n\leq N$ and assume that $\lambda_k^*$ takes the form of $(\ref{lamde_form})$ for all $k\geq n+1$. We shall show that $(\ref{lamde_form})$ also holds for $k=n$.
		Using $(\ref{lamde_form})$ and letting $k=n+1$, $\lambda_n^*$ can be written as
	\begin{align*}
		\lambda_n^*=Z_{n+1}\hat{x}_{n+1\mid n+1}^{R*}+X_{n+1}(\hat{x}_{n+1\mid n+1}^{L*}-\hat{x}_{n+1\mid n+1}^{R*})+S_{n+1}\mathbb{E}\hat{x}_{n+1\mid n+1}^{R*}.
	\end{align*}
	By using Lemma \ref{es}, one gets 
	\begin{align*}
		\hat{x}_{n+1\mid n+1}^{L*}&=\hat{x}_{n+1\mid n}^{L*}+W_{n+1}(y_{n+1}^{L*}-C\hat{x}_{n+1\mid n}^{L*})\nonumber\\
		&=A\hat{x}_{n\mid n}^{L*}+BU_n^*+B^L\tilde{u}_n^{L*}+W_{n+1}\big(C(Ax_n^*+BU_n^*+B^L\tilde{u}_n^{L*}+\omega_n) \nonumber\\
		&\quad +v_{n+1}-C(A\hat{x}_{n\mid n}^{L*}+BU_n^*+B^L\tilde{u}_n^{L*})\big)\nonumber\\
		&=A\hat{x}_{n\mid n}^{L*}+BU_n^*+B^L\tilde{u}_n^{L*}+W_{n+1}\big(CA(x_n^*-\hat{x}_{n\mid n}^{L*})+C\omega_n+v_{n+1}\big),\\
		\hat{x}_{n+1\mid n+1}^{R*}&=\eta_{n+1}\hat{x}_{n+1\mid n+1}^{L*}+(1-\eta_{n+1})\hat{x}_{n+1\mid n}^{R*} \nonumber \\
		&=\eta_{n+1}\Big(A(\hat{x}_{n\mid n}^{L*}-\hat{x}_{n\mid n}^{R*})+B^L\tilde{u}_n^{L*} \nonumber\\
		&\quad+W_{n+1}\big(CA(x_n^*-\hat{x}_{n\mid n}^{L*})+C\omega_n+v_{n+1}\big)\Big)+A\hat{x}_{n\mid n}^{R*}+BU_n^*,\\
		\hat{x}_{n+1\mid n+1}^{L*}-\hat{x}_{n+1\mid n+1}^{R*}&=(1-\eta_{n+1})\Big(A(\hat{x}_{n\mid n}^{L*}-\hat{x}_{n\mid n}^{R*})+B^L\tilde{u}_n^{L*}\nonumber\\
		&\quad +W_{n+1}\big(CA(x_n^*-\hat{x}_{n\mid n}^{L*})+C\omega_n+v_{n+1}\big)\Big).
	\end{align*}
	Thus,
	\begin{align}
		\lambda_n^*&=Z_{n+1}\Big((1-p)\Big(A(\hat{x}_{n\mid n}^{L*}-\hat{x}_{n\mid n}^{R*})+B^L\tilde{u}_n^{L*}+W_{n+1}\big(CA(x_n^*-\hat{x}_{n\mid n}^{L*})+C\omega_n+v_{n+1}\big)\Big)+A\hat{x}_{n\mid n}^{R*}\nonumber\\
		&\quad+BU_n^*\Big) 
		 +X_{n+1}\Big(p\Big(A(\hat{x}_{n\mid n}^{L*}-\hat{x}_{n\mid n}^{R*})+B^L\tilde{u}_n^{L*}+W_{n+1}\big(CA(x_n^*-\hat{x}_{n\mid n}^{L*})+C\omega_n+v_{n+1}\big)\Big)\Big)\nonumber\\
		&\quad + S_{n+1}\mathbb{E}[A\hat{x}_{n\mid n}^{R*}+BU_n^*].	\label{35}
	\end{align}
	Plugging $(\ref{35})$ into $(\ref{new-Uk})$, we get 
	\begin{align}\label{36}
		0=B^{\prime}Z_{n+1}(A\hat{x}_{n\mid n}^{R*}+BU_n^*)+B^{\prime}S_{n+1}A\mathbb{E}\hat{x}_{n\mid n}^{R*}+B^{\prime}S_{n+1}B\mathbb{E}U_n^{*}+RU_n^*.
	\end{align}
	Taking mathematical expectation on both sides of $(\ref{36})$, it yields that
	\begin{align}\label{37}
		\mathbb{E}U_n^*=-M_n^{-1}N_n \mathbb{E}\hat{x}_{n\mid n}^{R*}.
	\end{align}
	Taking $(\ref{37})$ into $(\ref{36})$, we get that the optimal controller $U_n^*$ is
	\begin{align}\label{38}
		U_n^*=-K_n\hat{x}^{R*}_{n\mid n}-\bar{K}_n\mathbb{E}\hat{x}^{R*}_{n\mid n}.
	\end{align}
	Plugging $(\ref{35})$ into $(\ref{new-tiu})$, we get
	\begin{align*}
		0&=(1-p)(B^L)^{\prime}Z_{n+1}\Big(A(\hat{x}_{n\mid n}^{L*}-\hat{x}_{n\mid n}^{R*})+B^L\tilde{u}_n^{L*}\Big)\nonumber\\&\quad +p(B^L)^{\prime}X_{n+1}(A(\hat{x}_{n\mid n}^{L*}-\hat{x}_{n\mid n}^{R*})+B^L\tilde{u}_n^{L*}).
	\end{align*} 
	Then, we have that the optimal controller $\tilde{u}_n^{L*}$ is 
	\begin{align}\label{40}
		\tilde{u}_n^{L*}=-\Lambda_n^{-1}L_n(\hat{x}_{n\mid n}^{L*}-\hat{x}_{n\mid n}^{R*}).
	\end{align}
	Now, we show that for $k=n$, $\lambda_{n-1}^*$ takes the form of $(\ref{lamde_form})$. Using $(\ref{38})$ and $(\ref{40})$, one has
	\begin{align*}
		\lambda_{n-1}^*&=\mathbb{E}[A^{\prime}\lambda_n^*+Q_{\mu}x_n^*-\mu Q \mathbb{E}x_n^*\mid \mathcal{F}_n^L]\nonumber\\
		&=(A^{\prime}\Theta_{n+1}A+Q_\mu)(\hat{x}_{n\mid n}^{L*}-\hat{x}_{n\mid n}^{R*})+(A^{\prime}Z_{n+1}A+Q_\mu)\hat{x}_{n\mid n}^{R*} \nonumber\\
		&\quad +(A^{\prime}S_{n+1}A-\mu Q)\mathbb{E}\hat{x}_{n\mid n}^{R*}+A^{\prime}Z_{n+1}BU_n^* +A^{\prime}\Theta_{n+1}B^L\tilde{u}_n^{L*}+A^{\prime}S_{n+1}B\mathbb{E}U_n^* \nonumber\\
		&=(A^{\prime}\Theta_{n+1}A+Q_\mu -A^{\prime}\Theta_{n+1}B^L\Lambda_n^{-1}L_n)(\hat{x}_{n\mid n}^{L*}-\hat{x}_{n\mid n}^{R*})\nonumber\\
		&\quad +(A^{\prime}Z_{n+1}A+Q_\mu-A^{\prime}Z_{n+1}BK_n)\hat{x}_{n\mid n}^{R*}\nonumber\\
		&\quad +(A^{\prime}S_{n+1}A+K_n^{\prime}\Upsilon_n^{-1}K_n-N_n^{\prime}M_n^{-1}N_n-\mu Q)\mathbb{E}\hat{x}_{n\mid n}^{R*} \nonumber\\
		&=Z_n\hat{x}_{n\mid n}^{R*}+X_n(\hat{x}_{n\mid n}^{L*}-\hat{x}_{n\mid n}^{R*})+S_n\mathbb{E}\hat{x}_{n\mid n}^{R*}.
	\end{align*}
	Thus $(\ref{lamde_form})$ holds for $k=n$.
	This completes the proof.
\end{proof}
\end{lemma}
The optimal control strategies for any fixed $\mu\geq 0$ are given in the theorem below.
\begin{theorem}\label{theorem2}
	For any fixed $\mu\geq 0$, Problem (FLQ) has a unique solution if and only if 
 $M_k>0$, $\Upsilon_k>0$ and $\Lambda_k>0$ for $k=N,\cdots,0$. In this case, the optimal controllers of Problem (FLQ) are given by 
	\begin{align}
		&u_k^{R^{*}}=-[0\quad I](K_k\hat{x}^{R*}_{k\mid k}+\bar{K}_k\mathbb{E}\hat{x}^{R*}_{k\mid k})\label{optimal_u_k^R},\\
		&u_k^{L^{*}}=-[I\quad 0](K_k\hat{x}^{R*}_{k\mid k}+\bar{K}_k\mathbb{E}\hat{x}^{R*}_{k\mid k})-\Lambda_k^{-1}L_k(\hat{x}^{L*}_{k\mid k}-\hat{x}^{R*}_{k\mid k}) \label{optimal_u_k^L},
	\end{align}
where $\hat{x}^{L*}_{k\mid k}$ and $\hat{x}^{R*}_{k\mid k}$ are defined in Lemma $\ref{lemma3}$, and $K_k$, $\bar{K}_k$, $\Lambda_k$, $L_k$ are given in $(\ref{DRE})$. Furthermore, the optimal cost function $\bar{J}^{*}(\mu)$ is
\begin{align}
	\bar{J}^{*}(\mu)=\mathbb{E}\Big\{x_0^{\prime}\big[Z_0\hat{x}_{0\mid 0}^{R}+X_0(\hat{x}_{0\mid 0}^{L}-\hat{x}_{0\mid 0}^{R})+S_0\mathbb{E}\hat{x}^{R}_{0\mid 0}\big]\Big\}+\sum_{k=0}^{N}C_k \label{optimal_J}
\end{align}
with
\begin{align}\label{C_k}
	C_k=&tr\Big\{(1-p)\Sigma_{k\mid k}^LA^{\prime}Z_{k+1}W_{k+1}CA+(1-p)Q_\omega Z_{k+1}W_{k+1}C+p\Sigma_{k\mid k}^LA^{\prime}C^{\prime}W_{k+1}X_{k+1}W_{k+1}CA \nonumber\\
	&+pQ_\omega C^{\prime}W_{k+1}X_{k+1}W_{k+1}C+pQ_vW_{k+1}X_{k+1}W_{k+1}+Q_\omega\Sigma_{k\mid k}^L\Big\}.
\end{align}
\begin{proof}

\textbf{"Necessity"}. Suppose Problem (FLQ) has a unique solution. We will show by induction that $M_k>0$, $\Upsilon_k>0$ and $\Lambda_k>0$ for $k=N,\cdots,0$.
For $k=0,\cdots,N$, define
\begin{align*}
	\bar{J}(k,\mu)&=\mathbb{E}\Big\{\sum_{l=k}^{N}x_l^\prime Q_{\mu}x_l-\mu x_l^\prime Q\mathbb{E}x_l+U_l^\prime RU_l+(\tilde{u}_l^{L})^{\prime}R^L\tilde{u}_l^{L}+x_{N+1}^\prime G_{\mu}x_{N+1}-\mu x_{N+1}^\prime Q\mathbb{E}x_{N+1}\Big\}.
\end{align*}
%
%
Firstly, for $k=N$, it is clear that $\bar{J}(N,\mu)$ can be expressed as a quadratic function of $x_N$, $\mathbb{E}x_N$, $U_N$, $\mathbb{E}U_N$ and $\tilde{u}_N^{L}$. Let $x_N=0$; since it is assumed that the problem admits a unique solution, $\bar{J}(N,\mu)$ must be strictly positive for any nonzero $U_N$ and $\tilde{u}_N^{L}$. For any $U_N\ne 0$ and $\tilde{u}_N^{L}\ne 0$, we have 
\begin{align}\label{bar_J_N}
	\bar{J}(N,\mu)=\mathbb{E}\Big\{\big(U_N-\mathbb{E}U_N\big)^{\prime}\Upsilon_N\big(U_N-\mathbb{E}U_N\big)+\mathbb{E}(U_N)^\prime M_N\mathbb{E}U_N+(\tilde{u}_N^{L})^{\prime}\Lambda_N\tilde{u}^{L}_N\Big\}> 0.
\end{align}
We immediately have $M_N>0$, $\Upsilon_N>0$ and $\Lambda_N>0$. In fact, in the case $\mathbb{E}U_N=0$, $U_N\ne 0$ and $\tilde{u}_N^{L}\ne 0$, $(\ref{bar_J_N})$ becomes 
\begin{align*}
	\bar{J}(N,\mu)=\mathbb{E}\Big\{U_N^{\prime}\Upsilon_NU_N+(\tilde{u}_N^{L})^{\prime}\Lambda_N\tilde{u}^{L}_N\Big\}> 0.
\end{align*}
Thus, we get $\Upsilon_N>0$ and $\Lambda_N>0$. On the other hand, if $U_N=\mathbb{E}U_N\ne 0$, (\ref{bar_J_N}) becomes
\begin{align*}
	\bar{J}(N,\mu)=\mathbb{E}\Big\{\mathbb{E}U_N^{\prime}M_N\mathbb{E}U_N+(\tilde{u}_N^{L})^{\prime}\Lambda_N\tilde{u}^{L}_N\Big\}> 0,
\end{align*}
which implies $M_N>0$ and $\Lambda_N>0$.

For any $n$ with $0\leq n\leq N$, assume that $M_k>0$, $\Upsilon_k>0$ and $\Lambda_k>0$ for all $k\geq n+1$. We shall show that $M_k>0$, $\Upsilon_k>0$ and $\Lambda_k>0$ for $k=n$. Note that 
\begin{align*}
	&\mathbb{E}[(x_k^*)^{\prime}\lambda_{k-1}^*-(x_{k+1}^*)^{\prime}\lambda_k^*] \nonumber\\
	&=\mathbb{E}\big[(x_k^*)^{\prime}\mathbb{E}[A^{\prime}\lambda_k^*+Q_{\mu}x_k^*-\mu Q\mathbb{E}x_k^*\mid \mathcal{F}_k^L]\big]-\mathbb{E}[(Ax_k^*+BU_k^*+B^L\tilde{u}_k^{L*}+\omega_k)\lambda_k^*] \nonumber\\
	&=\mathbb{E}[(x_k^*)^{\prime}Q_\mu x_k^*-\mu (x_k^*)^{\prime}Q\mathbb{E}x_k^*+(U_k^*)^{\prime}RU_k^*+(\tilde{u}_k^{L*})^{\prime}R^L\tilde{u}_k^{L*}]-\mathbb{E}[\omega_k^{\prime}\lambda_k^*]-tr(\Sigma_{k\mid k}^LQ_\mu).
\end{align*}
Taking summation from $k=n+1$ to $k=N$ on both sides of the above equality, it yields that
\begin{align*}
	&\mathbb{E}[(x_{n+1}^*)^{\prime}\lambda_{n}^*] \nonumber\\
	&=\sum_{k=n+1}^{N}\mathbb{E}[(x_k^*)^{\prime}Q_\mu x_k^*-\mu (x_k^*)^{\prime}Q\mathbb{E}x_k^*+(U_k^*)^\prime RU_k^*+(\tilde{u}_k^{L*})^{\prime}R^L\tilde{u}_k^{L*}+(x_{N+1}^*)^{\prime}G_\mu x_{N+1}^*-\mu (x_{N+1}^*)^{\prime}Q\mathbb{E}x_{N+1}^*]\nonumber\\
	&\quad-\mathbb{E}[\omega_k^{\prime}\lambda_k^*]-tr(\Sigma_{k\mid k}^LQ_\mu)-tr(\Sigma_{N+1\mid N+1}^LG_\mu).
\end{align*}
Thus, we have
\begin{align*}
	\min_{}\bar{J}(n,\mu)&=\min_{}\mathbb{E}[x_n^\prime Q_\mu x_n-\mu x_n^{\prime}Q\mathbb{E}x_n+U_n^{\prime}RU_n+(\tilde{u}_n^{L})^{\prime}R^L\tilde{u}_n^{L}]\\
	&\quad+ \mathbb{E}[(x_{n+1}^*)^{\prime}\lambda_n^*]+\sum_{k=n+1}^{N}\mathbb{E}[\omega_k^{\prime}\lambda_k^*]+tr(\Sigma_{k\mid k}^LQ_\mu)+tr(G_\mu\Sigma_{N+1\mid N+1}^L).
\end{align*}
Since $M_k>0$, $\Upsilon_k>0$ and $\Lambda_k>0$ for $k>n+1$, $(\ref{lamde_form})$ holds for $k=n+1$. Setting $x_n=0$, the above equation becomes
\begin{align*}
	&\quad\min_{}\bar{J}(n,\mu)\\
	&=\min_{}\mathbb{E}\Big\{\big(U_n-\mathbb{E}U_n\big)^{\prime}\Upsilon_n\big(U_n-\mathbb{E}U_n\big)+\mathbb{E}U_n^{\prime}M_n\mathbb{E}U_n+(\tilde{u}_n^{L})^{\prime}\Lambda_n\tilde{u}^{L}\Big\}+(1-p)tr(Q_\omega X_{n+1}W_{n+1}C)\\
	&\quad+ptr(Q_\omega Z_{n+1}W_{n+1}C)+\sum_{k=n+1}^{N}\mathbb{E}[\omega_k^{\prime}\lambda_k^*]+tr(\Sigma_{k\mid k}^LQ_\mu)+tr(G_\mu\Sigma_{N+1\mid N+1}^L).
\end{align*}
The uniqueness of the optimal $U_n^*$ and $\tilde{u}_n^{L*}$ implies that the terms $M_n>0$, $\Upsilon_n>0$ and $\Lambda_n>0$.
The proof of the necessity is completed.

\textbf{"Sufficiency"}. Suppose that $M_k>0$, $\Upsilon_k>0$ and $\Lambda_k>0$ for $k=N,\cdots,0$. The uniqueness of the solution to Problem (FLQ) is to be shown.
Define
\begin{align}
	V(k)=\mathbb{E}\Big\{x_k^{\prime}Z_k\hat{x}_{k\mid k}^R+x_k^{\prime}X_k(\hat{x}_{k\mid k}^L-\hat{x}_{k\mid k}^R)+x_k^{\prime}S_k\mathbb{E}x_k\Big\}.\label{V_k}
\end{align}
Then,
\begin{align}
	V_{k+1}&=\mathbb{E}\Big\{x_{k+1}^{\prime}Z_{k+1}\hat{x}_{k+1\mid k+1}^R+x_{k+1}^{\prime}X_{k+1}(\hat{x}_{k+1\mid k+1}^L-\hat{x}_{k+1\mid k+1}^R)+x_{k+1}^{\prime}S_{k+1}\mathbb{E}x_{k+1}\Big\}\nonumber\\
	&=\mathbb{E}\Big\{x_k^{\prime}(A^{\prime}Z_{k+1}A-K_k^{\prime}\Upsilon_k^{-1}K_k)x_k+\hat{x}^{R^{\prime}}_{k\mid k}K_k^{\prime}\Upsilon_k^{-1}K_k\hat{x}^{R}_{k\mid k} \nonumber\\
	&\quad+(\hat{x}_{k\mid k}^L-\hat{x}_{k\mid k}^R)^{\prime}\big((1-p)A^{\prime}Z_{k+1}A-A^{\prime}Z_{k+1}A+pA^{\prime}X_{k+1}A+K_k^{\prime}\Upsilon_k^{-1}K_k\big)(\hat{x}_{k\mid k}^L-\hat{x}_{k\mid k}^R) \nonumber\\
	&\quad+2\tilde{u}_k^{L^{\prime}}\big((1-p)B^{L^{\prime}}Z_{k+1}A+pB^{L^{\prime}}X_{k+1}A\big)(\hat{x}_{k\mid k}^L-\hat{x}_{k\mid k}^R) \nonumber\\
	&\quad+\tilde{u}_k^{L^{\prime}}\big((1-p)B^{L^{\prime}}Z_{k+1}B^{L}+pB^{L^{\prime}}X_{k+1}B^{L}\big)\tilde{u}_k^{L} \nonumber\\
	&\quad+2U_k^{\prime}B^{\prime}Z_{k+1}A\hat{x}_k^R+U_k^{\prime}B^{\prime}Z_{k+1}BU_k+x_{k+1}S_{k+1}\mathbb{E}x_{k+1} \nonumber\\
	&\quad+tr\Big\{\Sigma_{k\mid k}^LK_k^{\prime}\Upsilon_k^{-1}K_k+(1-p)\Sigma_{k\mid k}^LA^{\prime}Z_{k+1}W_{k+1}CA \nonumber\\
	&\quad+(1-p)Z_{k+1}W_{k+1}CQ_\omega+p\Sigma_{k\mid k}^LA^{\prime}C^{\prime}W_{k+1}X_{k+1}W_{k+1}CA \nonumber\\
	&\quad+pC^{\prime}W_{k+1}X_{k+1}W_{k+1}CQ_\omega+pW_{k+1}X_{k+1}W_{k+1}Q_v-\Sigma_{k\mid k}^LA^{\prime}Z_{k+1}A\Big\}. \label{V_k+1}
\end{align}
Combining $(\ref{V_k})$ with $(\ref{V_k+1})$, we get
\begin{align*}
	V(k)-V(k+1)=&\mathbb{E}\Big\{x_k^{\prime}Q_\mu x_k-\mu x_k^{\prime}Q\mathbb{E}x_k+U_k^{\prime}RU_k+\tilde{u}_k^{L^{\prime}}R^L\tilde{u}_k^L\\
	&-\big(U_k-\mathbb{E}U_k+K_k(x_k-\mathbb{E}x_k)\big)^{\prime}\Upsilon_k\big(U_k-\mathbb{E}U_k+K_k(x_k-\mathbb{E}x_k)\big)\\
	&-\big(\mathbb{E}U_k+(K_k+\bar{K}_k)\mathbb{E}x_k\big)^{\prime}M_k\big(\mathbb{E}U_k+(K_k+\bar{K}_k)\mathbb{E}x_k\big)\\
	&-\big(\tilde{u}_k^L+\Lambda_k^{-1}L_k(\hat{x}_{k\mid k}^L-\hat{x}_{k\mid k}^R)\big)^{\prime}\Lambda_k\big(\tilde{u}_k^L+\Lambda_k^{-1}L_k(\hat{x}_{k\mid k}^L-\hat{x}_{k\mid k}^R)\big)\Big\}-C_k.
\end{align*}
Taking summation from $k=0$ to $k=N$ on the both sides of above equation, the function $\bar{J}(\mu)$ can be written as
\begin{align*}
	\bar{J}(\mu)&=\mathbb{E}\Big\{x_0^{\prime}\big[Z_0\hat{x}_{0\mid 0}^R+X_0(\hat{x}_{0\mid 0}^L-\hat{x}_{0\mid 0}^R)+S_0\mathbb{E}x_0\big]\Big\}\\
	&+\sum_{k=0}^{N}\mathbb{E}\Big\{\big(U_k-\mathbb{E}U_k+K_k(x_k-\mathbb{E}x_k)\big)^{\prime}\Upsilon_k\big(U_k-\mathbb{E}U_k+K_k(x_k-\mathbb{E}x_k)\big)\\
	&+\big(\mathbb{E}U_k+(K_k+\bar{K}_k)\mathbb{E}x_k\big)^{\prime}M_k\big(\mathbb{E}U_k+(K_k+\bar{K}_k)\mathbb{E}x_k\big)\\
	&+\big(\tilde{u}_k^L+\Lambda_k^{-1}L_k(\hat{x}_{k\mid k}^L-\hat{x}_{k\mid k}^R)\big)^{\prime}\Lambda_k\big(\tilde{u}_k^L+\Lambda_k^{-1}L_k(\hat{x}_{k\mid k}^L-\hat{x}_{k\mid k}^R)\big)+C_k\Big\},
\end{align*}
where $C_k$ is given in (\ref{C_k}). Noticing $M_k>0$, $\Upsilon_k>0$ and $\Lambda_k>0$ for $k=N,\cdots,0$, we have
\begin{align*}
	\bar{J}(\mu)\geq \mathbb{E}\Big\{x_0^{\prime}\big[Z_0\hat{x}_{0\mid 0}^R+X_0(\hat{x}_{0\mid 0}^L-\hat{x}_{0\mid 0}^R)+S_0\mathbb{E}x_0\big]\Big\}+\sum_{k=0}^{N}C_k;
\end{align*}
thus the minimum of $\bar{J}(\mu)$ is given by ($\ref{optimal_J}$). In this case the optimal controller  will satisfy
\begin{align*}
	U_k^*-\mathbb{E}U_k^*+K_k(x_k^*-\mathbb{E}x_k^*)&=0,\\
	\mathbb{E}U_k^*-(K_k+\bar{K}_k)\mathbb{E}x_k^*&=0,\\
	\tilde{u}_k^{L*}+\Lambda_k^{-1}L_k(\hat{x}_{k\mid k}^{L*}-\hat{x}_{k\mid k}^{R*})&=0.
\end{align*}
Therefore, the optimal controller for Problem (FLQ) can be uniquely obtained as $(\ref{optimal_u_k^R})$  $(\ref{optimal_u_k^L})$.
\end{proof}
\end{theorem}
We know that in addition to Gaussian noise, there are some noises that are more prone to extreme situations, such as heavy-tailed or skewed noise (\cite{Tsiamis20}). The following assumption and corollary give the relevant conclusions of optimal control when the system noise is non-Gaussian.

\begin{assumption}\label{no_G}
	The process noise $\omega_k$ and measurement noise $v_k$ are i.i.d. across time, but not necessarily Gaussian.
\end{assumption}
For non-Gaussian noise, we cannot use standard Kalman filter. For this reason, define the error between the prediction and the estimate for local and remote controller:
\begin{align}
	e_k^L=\hat{x}_{k\mid k}^L-\hat{x}_{k\mid k-1}^L,\label{non_G_L}\\
	e_k^R=\hat{x}_{k\mid k}^R-\hat{x}_{k\mid k-1}^R\label{non_G_R}.
\end{align}
\begin{corollary}
	Under Assumption \ref{no_G}, the control policy and the estimation process can be designed independent of estimation process. In the other words, the certainty equivalence property hold.
\end{corollary}
\begin{proof}
	Under Assumption \ref{no_G}, we know that $\mathbb{E}(e_{k+1}^R\mid \mathcal{F}^R_k)=\mathbb{E}(e_{k+1}^L\mid \mathcal{F}^R_k)=0$ from $(\ref{non_G_L})$ and $(\ref{non_G_R})$. Obviously, the derivation of optimal risk constraint control is similar to Lemma \ref{lemma3} and Theorem \ref{theorem2}, so it is omitted. The certainty equivalence property holds.
\end{proof}

\begin{remark}
In the above, we have assumed that the state constrains  matrix is same to the state weighting matrix $Q$ in the cost function $J$.  However, all of our derived results are still valid if we redefine the matrix $Q$ in constraint (\ref{constraint-0}) to make it different from the  state weighting matrix in the cost.
\end{remark}

\subsection{Infinite-Horizon case}
In this section, the infinite-horizon version of Problem (FLQ) is studied.

\textbf{Problem ($\mbox{FLQ}_\infty$).} For every fixed multiplier $\mu\geq 0$, find $\mathcal{F}^R_k$-measurable $U_k^*$ and $\mathcal{F}^L_k$-measurable $\tilde{u}_k^{L*}$ that make system (\ref{new-system-form}) is bounded in the mean-square sense and simultaneously minimize the function $\bar{J}_\infty(\mu)$:
\begin{align}
	\min\limits_{} &\bar{J}_{\infty}(\mu)=\lim\limits_{N\to \infty} \frac{1}{N}\mathbb{E}\Big\{\sum\limits_{k=0}^{\infty}x_k^{\prime}Q_{\mu}x_k-\mu \mathbb{E}x_k^{\prime}Q\mathbb{E}x_k+U_k^{\prime}RU_k+(\tilde{u}_k^{L})^\prime R^L\tilde{u}_k^L\Big\}\label{inf_J}.
\end{align}

 Because of the additive noise, it is impossible that the controlled system achieves the mean-square stability. Alternatively, we will study the mean-square boundedness
and investigate the corresponding necessary and sufficient conditions.
We first consider the system without additive noise with observation models (\ref{obervation-1})
\begin{align}\label{inf_system}
	x_{k+1}=Ax_k+BU_k+B^L\tilde{u}_k^L,
\end{align}
and the corresponding infinite-horizon cost function (for every fixed multiplier $\mu\geq 0$) is given by 
\begin{align}
	&{J}_{\infty}(\mu)= \mathbb{E}\Big\{\sum\limits_{k=0}^{\infty}x_k^{\prime}Q_{\mu}x_k-\mu \mathbb{E}x_k^{\prime}Q\mathbb{E}x_k+U_k^{\prime}RU_k+(\tilde{u}_k^{L})^\prime R^L\tilde{u}_k^L\Big\}\label{inf_new_problem}.
\end{align}
%


\begin{definition}
	The system (\ref{inf_system}) with $U_k=0$ and $\tilde{u}_k^L=0$ is said to be asymptotically mean-square stable, if for any initial value $x_0$ there holds $\lim\limits_{k\to \infty}\mathbb{E}[x_k^{\prime}x_k]=0$.
\end{definition}
\begin{definition}
	The system (\ref{inf_system}) is said to be stabilizable in the mean-square sense, if there exist $\mathcal{F}_k^R$-measurable $U_k=-K\hat{x}_{k\mid k}^R-\bar{K}\mathbb{E}\hat{x}_{k\mid k}^R$ and $\mathcal{F}_k^L$-measurable $\tilde{u}_k^L=-L^L(\hat{x}_{k\mid k}^L-\hat{x}_{k\mid k}^R)$ with constant matrices $K$, $\bar{K}$ and $L^L$ such that for any $x_0$ the closed-loop system of (\ref{inf_system}) is asymptotically mean-square stable.
\end{definition}

Next we denote $\mathcal{Q}=\mbox{diag}\{Q_\mu,Q\}$ and present the following assumptions. 
\begin{assumption}\label{ass2}
	$R^L> 0$, $R^R> 0$ and $\mathcal{Q}=D^{\prime}D\geq 0$  for some matrices $D$.
\end{assumption}
\begin{assumption}\label{ass3}
	$(A,\mathcal{Q}^{\frac{1}{2}})$ is observable, and $(A,C)$ is detectable. 
\end{assumption}
%

\begin{lemma}\label{lemma4}
	Under Assumptions \ref{ass2} and \ref{ass3}, if there exist stabilizing controllers
	\begin{align}
		U_k&=-K\hat{x}^R_{k\mid k}-\bar{K}\mathbb{E}\hat{x}^R_{k\mid k},\\
		\tilde{u}_k^L&=-\Lambda^{-1}L(\hat{x}^L_{k\mid k}-\hat{x}^R_{k\mid k})
	\end{align}
	to make the system (\ref{inf_system}) stabilizable in the mean-square sense, then the following algebraic Riccati equations
	\begin{align}
		Z&=A^{\prime}ZA+Q_\mu-K^{\prime}\Upsilon K,\label{inf_Z}\\
		X&=A^{\prime}\Theta A+Q_\mu-L^{\prime}\Lambda^{-1}L,\label{inf_X}\\
		S&=A^{\prime}SA-\mu Q+K^{\prime}\Upsilon K-N^{\prime}M^{-1}N\label{inf_S}
	\end{align}
	 admit solutions $Z$, $X$ and $S$ that satisfy $Z>0$, $Z+S>0$ and $\Theta>0$
	with
	\begin{equation}\label{inf_DRE}
		\left\{
		\begin{aligned}
			\Upsilon&=B^{\prime}ZB+R,\\
			K&=\Upsilon^{-1}B^{\prime}ZA,\\
			\bar{K}&=M^{-1}N-\Upsilon^{-1}B^{\prime}ZA, \\
			M&=B^{\prime}ZB+B^{\prime}SB+R,\\
			N&=B^{\prime}ZA+B^{\prime}SA,\\ 
			\Lambda&=B^{L^{\prime}}\Theta B^L+R^L,\\
			L&=B^{L^{\prime}}\Theta A,\\
			\Theta&=(1-p)Z+pX .
		\end{aligned}
		\right.
	\end{equation}
\end{lemma}
\begin{proof}
	
	To make the time horizon $N$ explicit in the finite horizon case, we rewrite $Z_{k}$, $X_{k}$, $S_k$, $\Theta_k$, $\Upsilon_{k}$, $K_k$, $\bar{K}_k$, $M_k$, $N_k$, $\Lambda_k$, $L_k$, $\bar{J}(\mu)$ of (\ref{Z})-(\ref{DRE}) and (\ref{bar_J}) as $Z_{k}(N)$, $X_{k}(N)$, $S_k(N)$, $\Theta_k(N)$, $\Upsilon_{k}(N)$, $K_k(N)$, $\bar{K}_k(N)$, $M_k(N)$, $N_k(N)$, $\Lambda_k(N)$, $L_k(N)$ and $\bar{J}_N(\mu)$, respectively.

	First, we shall show $\Theta_0(N)$ and $Z_0(N)+S_0(N)$ are monotonically increasing with $N$. 
	From (\ref{obervation-1}) and (\ref{Z})-(\ref{S}), we know that the solutions to (\ref{Z})-(\ref{S}) are irrespective to the observation model.
	Hence, in our setting, it is valid to set $C=I$ and $v_{k}=0$, and (\ref{obervation-1}) becomes $y_{k}^{L}=x_{k}$, which implies $\hat{x}_{k \mid k}^{L}=x_{k}, \Sigma_{k \mid k}^{L}=0$. 
	Similarly, the equations (\ref{Z})-(\ref{S}) are irrespective with the initial value $x_0$, and we set $\bar{x}_0=0$ here. Then, (\ref{optimal_J}) becomes
	\begin{align}
		\bar{J}_{N}^{*}(\mu) & =\mathbb{E}\left\{x_0^{\prime}\big[Z_0(N)\hat{x}_{0\mid 0}^R+X_0(N)(\hat{x}_{0\mid 0}^L-\hat{x}_{0\mid 0}^R)+S_0(N)\mathbb{E}\hat{x}^R_{0\mid 0}\big]\right\}\nonumber \\
		& =\mathbb{E}\left[\eta_{0} x_{0}^{\prime} Z_{0}^{}(N) x_{0}+(1-\eta_{0})x_{0}^{\prime} X_{0}(N)x_0\right] \nonumber\\
		& =\mathbb{E}\left[x_{0}^{\prime} \Theta_{0}(N) x_{0}\right] \geq 0\label{x_0_0}.
	\end{align}
	Actually, since $\bar{J}_N^{*}(\mu)=\mathbb{E}\left[x_{0}^{\prime} \Theta_{0}(N) x_{0}\right] \leq \mathbb{E}\left[x_{0}^{\prime} \Theta_{0}(N+1)x_{0}\right]=\bar{J}_{N+1}^{*}(\mu)$ and the initial value $x_{0}$ is arbitrary, we can obtain that  $\Theta_{0}(N)\geq 0$ increases with respect to $N$. On the other hand, if $x_0=\mathbb{E}x_0$, i.e., $x_0$ is deterministic, we know that
	\begin{align}
		\bar{J}_{N}^{*}(\mu)=x_{0}^{\prime}(Z_0(N)+S_0(N))x_{0}\geq 0 \label{x_deter};
	\end{align}
	this yields that $Z_0(N)+S_0(N)$ also increases with respect to $N$.

	Next we shall show that $\Theta_{0}(N)$ and $Z_0(N)+S_0(N)$ are bounded. 
	From (\ref{inf_system}), we have
	\begin{align*}
		&\lim \limits_{k \rightarrow \infty} \mathbb{E}[x_{k}^{\prime} x_{k}]\\
		&=\lim \limits_{k \rightarrow \infty} \mathbb{E}\Big\{[(x_ {k}-\hat{x}_{k|k}^{L})+(\hat{x }_{k|k}^{L}-\hat{x}_{k|k}^{R})+ \hat{x}_{k|k}^{R}]^{\prime}[(x_ {k}-\hat {x}_{k|k}^{L})+(\hat{x }_{k|k}^{L}-\hat{x}_{k|k}^{R})+ \hat{x}_{k|k}^{R}]\Big\}\\
		&=\lim \limits_{k \rightarrow \infty} \Big[tr(\Sigma_{k \mid k}^{L})+\mathbb{E}(\hat{x}_{k \mid k}^{L}-\hat{x}_{k \mid k}^{R})^{\prime}(\hat{x}_{k \mid k}^{L}-\hat{x}_{k \mid k}^{R})+\mathbb{E}(\hat{x}_{k \mid k}^{R^{\prime}} \hat{x}_{k \mid k}^{R})\Big]\\
		&=0.
	\end{align*}
	Accordingly, we have $\lim\limits_{k\to\infty} \mathbb{E}(\hat{x}_{k \mid k}^{L}-\hat{x}_{k \mid k}^{R})^{\prime}(\hat{x}_{k \mid k}^{L}-\hat{x}_{k \mid k}^{R})=0$ and $\lim\limits_{k\to\infty}\mathbb{E}(\hat{x}_{k \mid k}^{R^{\prime}} \hat{x}_{k \mid k}^{R})=0$. From \cite{Bou99}, there exist constants $c_{1}>0, c_{2}>0$ and $c_{3}>0$ such that 
	\begin{align*}
	&\sum\limits_{k=0}^{\infty} \mathbb{E}\left(x_{k}^{\prime} x_{k}\right) \leq c_{1} \mathbb{E}\left(x_{0}^{\prime} x_{0}\right), \\ 
	&\sum\limits_{k=0}^{\infty} \mathbb{E}(\hat{x}_{k \mid k}^{R^{\prime}} \hat{x}_{k \mid k}^{R}) \leq c_{2} \mathbb{E}(\hat{x}_{0 \mid 0}^{R^{\prime}} \hat{x}_{0 \mid 0}^{R}), \\
	 &\sum\limits_{k=0}^{\infty}\mathbb{E}(\hat{x}_{k \mid k}^{L}-\hat{x}_{k \mid k}^{R})^{\prime}(\hat{x}_{k \mid k}^{L}-\hat{x}_{k \mid k}^{R})\leq c_{3} \mathbb{E}(\hat{x}_{0 \mid 0}^{L}-\hat{x}_{0 \mid 0}^{R})^{\prime}(\hat{x}_{0 \mid 0}^{L}-\hat{x}_{0 \mid 0}^{R}).
	\end{align*}
	Since $Q_\mu\geq 0$, $Q\geq 0$, $R> 0$ and $R^{L}> 0$,  there exists constant $\lambda$ such that 
	\begin{align*}
	&\begin{bmatrix}
		Q_\mu&0\\0&Q
	\end{bmatrix}\leq \lambda I, \\
	&\begin{bmatrix}
		K^{\prime}RK&0\\0&(K+\bar{K})^{\prime}R(K+\bar{K})
	\end{bmatrix}\leq \lambda I,
	\end{align*}
	and $L^{\prime}\Lambda^{-1^{\prime}}R^{L}\Lambda^{-1}L\leq \lambda I$. 
	Hence, we have
	\begin{align*}
		{J}_{\infty}(\mu)&= \sum\limits_{k=0}^{\infty}\mathbb{E}\Big\{x_k^{\prime}Q_{\mu}x_k-\mu \mathbb{E}x_k^{\prime}Q\mathbb{E}x_k+U_k^{\prime}RU_k^L+\tilde{u}_k^{L^\prime}R^L\tilde{u}_k^L\Big\}\\
		&=\sum\limits_{k=0}^{\infty}\mathbb{E}\Bigg\{\begin{bmatrix}
			x_k-\mathbb{E}x_k\\\mathbb{E}x_k
		\end{bmatrix}^{\prime}\begin{bmatrix}
			Q_\mu&0\\0&Q
		\end{bmatrix}\begin{bmatrix}
			x_k-\mathbb{E}x_k\\\mathbb{E}x_k
		\end{bmatrix}+(\hat{x}_{k\mid k}^L-\hat{x}_{k\mid k}^R)^{\prime}L^{\prime}\Lambda^{-1^{\prime}}R^{L}\Lambda^{-1}L(\hat{x}_{k\mid k}^L-\hat{x}_{k\mid k}^R)\\&\quad +
		\begin{bmatrix}
			\hat{x}_{k\mid k}^R-\mathbb{E}\hat{x}_{k\mid k}^R\\\mathbb{E}\hat{x}_{k\mid k}^R
		\end{bmatrix}^{\prime}\begin{bmatrix}
			K^{\prime}RK&0\\0&(K+\bar{K})^{\prime}R(K+\bar{K})
		\end{bmatrix}\begin{bmatrix}
			\hat{x}_{k\mid k}^R-\mathbb{E}\hat{x}_{k\mid k}^R\\\mathbb{E}\hat{x}_{k\mid k}^R
		\end{bmatrix}\Bigg\}\\
		&\leq \sum\limits_{k=0}^{\infty}\lambda\mathbb{E}\Bigg\{\begin{bmatrix}
			x_k-\mathbb{E}x_k\\\mathbb{E}x_k
		\end{bmatrix}^{\prime}\begin{bmatrix}
			x_k-\mathbb{E}x_k\\\mathbb{E}x_k
		\end{bmatrix}+(\hat{x}_{k\mid k}^L-\hat{x}_{k\mid k}^R)^{\prime}(\hat{x}_{k\mid k}^L-\hat{x}_{k\mid k}^R)\\
		&\hphantom{\leq}\,\, +
		\begin{bmatrix}
			\hat{x}_{k\mid k}^R-\mathbb{E}\hat{x}_{k\mid k}^R\\\mathbb{E}\hat{x}_{k\mid k}^R
		\end{bmatrix}^{\prime}\begin{bmatrix}
			\hat{x}_{k\mid k}^R-\mathbb{E}\hat{x}_{k\mid k}^R\\\mathbb{E}\hat{x}_{k\mid k}^R
		\end{bmatrix}\Bigg\}\\
		&\leq
		\lambda \big\{c_1\mathbb{E}(x_0^{\prime}x_0)+c_2\mathbb{E}(\hat{x}_{0 \mid 0}^{R^{\prime}}\hat{x}_{0 \mid 0}^R)+c_3\mathbb{E}(\hat{x}_{0 \mid 0}^L-\hat{x}_{0 \mid 0}^R)^{\prime}(\hat{x}_{0 \mid 0}^L-\hat{x}_{0 \mid 0}^R)\big\}\triangleq M.
	\end{align*}
	Thus, with (\ref{x_0_0}) we obtain 
	$\bar{J}_{N}^{*}(\mu)=\mathbb{E}\left[x_{0}^{\prime} \Theta_{0}(N) x_{0}\right]\leq J_{\infty}\leq M$, which implies the boundedness of $\Theta_{0}(N)$. Similarly, let the initial state be arbitrary deterministic, we can get $Z_{0}(N)+S_{0}(N)$ is bounded too. Hence, $\Theta_{0}(N)$ and $Z_{0}(N)+S_{0}(N)$ are convergent, i.e., there exists $\Theta$, $Z+S$ such that
	\begin{align*}
		&\lim_{k\to \infty}\Theta_{k}(N)=\lim_{k\to \infty}\Theta_{0}(N-k)=\Theta\geq 0,\\
		&\lim_{k\to \infty}Z_{k}(N)+S_k(N)=\lim_{k\to \infty}Z_{0}(N-k)+S_{0}(N-k)=Z+S \geq 0.
	\end{align*}
	Then, we shall prove that $Z_{k}(N)$, $X_{k}(N)$ and $S_{k}(N)$ are convergent. We set the channel failure probability  $p=0$ because $Z_k(N)$ is irrespective with $p$. Hence, we get $Z_k(N)$ is convergent due to the convergence of $\Theta_{k}(N)$. Accordingly, we obtain that $X_k(N)$ and $S_k(N)$ are convergent due to the convergence of $Z_k(N)$, $\Theta_{k}(N)$ and $Z_k(N)+S_k(N)$. Furthermore, in view of (\ref{DRE}), we know that $\Upsilon_{k}(N)$, $K_k(N)$, $\bar{K}_k(N)$, $M_k(N)$, $N_k(N)$, $\Lambda_k(N)$ and $L_k(N)$ are convergent.
	
	Finally, we shall show that $Z>0$, $Z+S>0$ and $\Theta>0$. We first prove $\Theta>0$, i.e., there exists $m>0$ satisfying $\Theta_0(m)>0$. 
	Assuming that the situation is not true, then there exists $x_0\neq0$, $\bar{x}_0=0$ such that
	\begin{align*}
	 \bar{J}_N^{*}(\mu)&=\sum\limits_{k=0}^{N}\mathbb{E}\big[(x_k^{*})^\prime Q_\mu x_k^{*}-\mu\mathbb{E}(x_k^{*})^\prime Q\mathbb{E}x_k^{*}+(U_k^{*})^{\prime}RU_k^{*}+(\tilde{u}_k^{L^{*}})^{\prime} R\tilde{u}_k^{L^*}\big]\\
	 &=\sum\limits_{k=0}^{N}\mathbb{E}\bigg\{\begin{bmatrix}
		x_k^*-\mathbb{E}x_k^*\\\mathbb{E}x_k^*
	\end{bmatrix}^{\prime}\begin{bmatrix}
		Q_\mu&0\\0&Q
	\end{bmatrix}\begin{bmatrix}
		x_k^*-\mathbb{E}x_k^*\\ \mathbb{E}x_k^*
	\end{bmatrix}+(U_k^{*})^{\prime}RU_k^{*}+(\tilde{u}_k^{L^{*}})^{\prime}R\tilde{u}_k^{L^*}\bigg\}\\
	&=\mathbb{E}\left[x_{0}^{\prime} \Theta_{0}(N) x_{0}\right]\\
	&= 0,
	\end{align*}
	where $x_k^{*}$, $U_k^{*}$ and $\tilde{u}_k^{L^*}$ are optimal state and optimal controllers. From Assumption \ref{ass2}, we get $D\begin{bmatrix}
		x_k^*-\mathbb{E}x_k^*\\\mathbb{E}x_k^*
	\end{bmatrix}=0$, $U_k^{*}=0$ and $\tilde{u}_k^{L^*}=0$. With Assumption \ref{ass3}, we obtain $x_0=x_k^*=0$, which contradicts the hypothesis. Hence, we get $\Theta>0$. On the other hand, we set the initial state be arbitrary deterministic, then we can obtain $Z+S>0$. Hence, if we set $p=0$, $Z>0$ is obtained. This completes the proof.
\end{proof}

In what follows, we shall consider Problem ($\mbox{FLQ}_\infty$), and first make the following  assumption.

\begin{assumption}\label{ass4}
	$(A,Q_\omega)$ is stabilizable.
\end{assumption}
\begin{lemma}\label{lemma5}
	For system (\ref{new-system-form}) with arbitrary initial value of state and under Assumption \ref{ass3} and  Assumption \ref{ass4}, the estimator error covariance matrix  $\Sigma_{k \mid k}^L=\mathbb{E}[(x_k-\hat{x}_{k\mid k}^L)(x_k-\hat{x}_{k\mid k}^L)^{\prime}]$ is asymptotic bounded, i.e., $\lim\limits_{k\to \infty}\Sigma_{k \mid k}^L=\Sigma^L$. 
	Furthermore, $\Sigma_{k \mid k}^R=\mathbb{E}[(x_k-\hat{x}_{k\mid k}^R)(x_k-\hat{x}_{k\mid k}^R)^{\prime}]$ is asymptotic bounded, i.e., $\lim\limits_{k\to \infty}\Sigma_{k \mid k}^R=\Sigma^R$, if and only if $\sqrt{p}|\lambda_{max}(A-B^L\Lambda^{-1}L)|<1$,  where $\lambda_{max}(A-B^L\Lambda^{-1}L)$ is the eigenvalue of matrix $A-B^L\Lambda^{-1}L$ with the largest absolute value.
	\begin{proof}
		The convergence of $\Sigma_{k \mid k}^L$ can be easily obtained from \cite{Maybeck80} under Assumption \ref{ass3} and Assumption \ref{ass4}, i.e., $\lim\limits_{k\to \infty}\Sigma_{k \mid k}^L=\Sigma^L$. From (\ref{es_x_R}), we get
		\begin{align*}
			x_k-\hat{x}_{k\mid k}^R&=x_k-\eta_k \hat{x}_{k\mid k}^L-(1-\eta_{k})\hat{x}_{k\mid k-1}^R\\
			&=x_k-\eta_k x_k+\eta_k x_k-\eta_k \hat{x}_{k\mid k}^L-(1-\eta_{k})\hat{x}_{k\mid k-1}^R\\
			&=(1-\eta_k)(x_k-\hat{x}_{k\mid k-1}^R)+\eta_k(x_k-\hat{x}_{k\mid k}^L).
		\end{align*}
		Then, with (\ref{new-system-form}), (\ref{optimal_u_k^R}), (\ref{optimal_u_k^L}) and Lemma \ref{es}, we have 
		\begin{align*}
			\Sigma_{k \mid k}^R&=\mathbb{E}\Big\{(1-\eta_k)^2\big((A-B^L\Lambda_k^{-1}L)\Sigma_{k-1\mid k-1}^R(A-B^L\Lambda_k^{-1}L)^{\prime}\\
			&\quad+B^L\Lambda_k^{-1}L\Sigma_{k-1 \mid k-1}^L(A-B^L\Lambda_k^{-1}L)^{\prime}+A\Sigma_{k-1 \mid k-1}^L(B^L\Lambda_k^{-1}L)^{\prime}\big)+\eta_{k}^2\Sigma_{k \mid k}^L\\
			&\quad+(1-\eta_k)\eta_kA\Sigma_{k-1 \mid k-1}^L(A-W_kCA)^{\prime}+(1-\eta_k)\eta_k(A-W_kCA)\Sigma_{k-1 \mid k-1}^LA^{\prime}\Big\}+pQ_{\omega}\\
			&=\sqrt{p}\big((A-B^L\Lambda_k^{-1}L)\Sigma_{k-1\mid k-1}^R(A-B^L\Lambda_k^{-1}L)^{\prime}\sqrt{p}\\
			&\quad+pB^L\Lambda_k^{-1}L\Sigma_{k-1 \mid k-1}^L(A-B^L\Lambda_k^{-1}L)^{\prime}+pA\Sigma_{k-1 \mid k-1}^L(B^L\Lambda_k^{-1}L)^{\prime}\big)+(1-p)\Sigma_{k \mid k}^L+pQ_{\omega}.
		\end{align*}
	As $\lim\limits_{k\to \infty}\Sigma_{k \mid k}^L=\Sigma^L$,  it can be readily obtained that $\lim\limits_{k\to \infty}\Sigma^R_{k\mid k}=\Sigma^R$ if and only if $\sqrt{p}|\lambda_{max}(A-B^L\Lambda^{-1}L)|<1$.
	\end{proof}
\end{lemma}
We are now in the position to present the main results of this section.
\begin{theorem}
	Under Assumption \ref{ass2}, Assumption \ref{ass3} and Assumption \ref{ass4}, the system (\ref{new-system-form}) is bounded in the mean-square sense, if and only if there exist solutions $Z$, $X$, $S$ to (\ref{inf_Z}), (\ref{inf_X}) and (\ref{inf_S}) such that $Z>0$, $Z+S>0$, and $\Theta>0$ and $\sqrt{p}|\lambda_{max}(A-B^L\Lambda^{-1}L)|<1$.
	In this case, the optimal stabilizing controllers are as
	\begin{align}
		U_k^*&=-K\hat{x}^{R*}_{k\mid k}-\bar{K}\mathbb{E}\hat{x}^{R*}_{k\mid k}\label{inf_U},\\
		\tilde{u}_k^{L*}&=-\Lambda^{-1}L(\hat{x}^{L*}_{k\mid k}-\hat{x}^{R*}_{k\mid k})\label{inf_u_L}.
	\end{align}
\end{theorem}
\begin{proof}
	\textbf{"Sufficiency"}. Under Assumption \ref{ass2}-\ref{ass4}, if there exist solutions $Z$, $X$, $S$ to (\ref{inf_Z}), (\ref{inf_X}) and (\ref{inf_S}) such that $Z>0$, $Z+S>0$ and $\Theta>0$, $\sqrt{p}|\lambda_{max}(A-B^L)|<1$, we shall show that (\ref{new-system-form}) is bounded in the mean-square sense under controllers (\ref{inf_U}) and (\ref{inf_u_L}). Noting (\ref{new-system-form}), (\ref{inf_U}) and (\ref{inf_u_L}), we can get
	\begin{align}
		&x_{k+1}-\mathbb{E}x_{k+1}\nonumber=(A-BK)(x_k-\mathbb{E}x_k)+BK(x_k-\hat{x}_{k\mid k}^R)-B^L\Lambda_k^{-1}L(\hat{x}_{k\mid k}^L-\hat{x}_{k\mid k}^R)+\omega_k.\nonumber
	\end{align}
	Then,
	\begin{align*}
		&\mathbb{E}\big[(x_{k+1}-\mathbb{E}x_{k+1})^{\prime}(x_{k+1}-\mathbb{E}x_{k+1})\big]\\
		&=\mathbb{E}\big[(x_k-\mathbb{E}x_k)^{\prime}(A-BK)^{\prime}(A-BK)(x_k-\mathbb{E}x_k)\big]\\
		&\quad+tr\Big\{\Sigma_{k\mid k}^{R}\big((A-BK)^{\prime}(BK-B^{L}\Lambda^{-1}L)+K^{\prime}B^{\prime}A-(B^{L}\Lambda^{-1}L)^{\prime}(A-BK)\big)\Big\}\\
		&\quad+tr\Big\{(\Sigma_{k\mid k}^{R}-\Sigma_{k\mid k}^{L})\big((B^{L}\Lambda^{-1}L)^{\prime}(B^{L}\Lambda^{-1}L-BK)+K^{\prime}B^{\prime}B^{L}\Lambda^{-1}L\big)\Big\}+Q_\omega,\\
		&\mathbb{E}x_{k+1}^{\prime}\mathbb{E}x_{k+1}\\
		&=\mathbb{E}x_k^{\prime}(A-BM^{-1}N)^{-1}(A-BM^{-1}N)\mathbb{E}x_k.
	\end{align*}
	Hence, with the convergence of $\Sigma_{k\mid k}^{R}$ and $\Sigma_{k\mid k}^{L}$, it can be known that $\lim\limits_{k\to \infty}\mathbb{E}(x_{k+1}^{\prime}x_{k+1})=\lim\limits_{k\to \infty}\mathbb{E}\big[(x_{k+1}-\mathbb{E}x_{k+1})^{\prime}(x_{k+1}-\mathbb{E}x_{k+1})\big]+\lim\limits_{k\to \infty}\mathbb{E}x_{k+1}^{\prime}\mathbb{E}x_{k+1}$ is bounded in the mean-square sense if and only if the following linear systems:
	\begin{align}
		\alpha_{k+1}&=(A-BK)\alpha_{k},\label{alpha}\\
		\varphi_{k+1}&=(A-BM^{-1}N)\varphi_{k}\label{varphi}
	\end{align}
	with the initial state $\alpha_0=\varphi_0=x_0$, are stable in the mean-square sense.
	To this end, we rewrite (\ref{inf_Z}) and (\ref{inf_S}) as
	\begin{align}
		Z&=K^{\prime}RK+Q_\mu+(A-BK)^{\prime}Z(A-BK),\label{re_Z}\\
		Z+S&=N^{\prime}M^{-1^{\prime}}RM^{-1}N+Q+(A-BM^{-1}N)^{\prime}(Z+S)(A-BM^{-1}N)\label{re_Z_S}.
	\end{align}
	Now we shall present that (\ref{alpha}) and (\ref{varphi}) are stable in the mean-square sense.
	Letting the Lyapunov functions $W^1_k=\mathbb{E}(\alpha^{\prime}_kZ\alpha)$, $W^2_k=\mathbb{E}\big[\varphi^{\prime}_k(Z+S)\varphi_k\big]$ and noting ($\ref{re_Z}$) and ($\ref{re_Z_S}$), we can obtain
	\begin{align*}
		W_{k+1}^1-W_k^1&=\mathbb{E}\Big\{\alpha_k^{\prime}[(A-BK)^{\prime}Z(A-BK)-Z]\alpha_k\Big\}\\
		&=-\mathbb{E}\Big\{\alpha_k^{\prime}(K^{\prime}RK+Q_\mu)\alpha_k\Big\},\\
		W_{k+1}^2-W_k^2&=\mathbb{E}\Big\{\varphi_k^{\prime}[(A-BM^{-1}N)^{\prime}(Z+S)(A-BM^{-1}N)-(Z+S)]\varphi_k\Big\}\\
		&=-\mathbb{E}\Big\{\varphi_k^{\prime}(N^{\prime}M^{-1^{\prime}}RM^{-1}N+Q)\varphi_k\Big\};
	\end{align*}
	this means that $W_k^1$ and $W_k^2$ decrease with respect to $k$. Owing to the semi-definite positiveness of $Z$ and $Z+S$, it can be known that $W_k^1$ and $W_k^2$ are bounded below. Hence, $W_k^1$ and $W_k^2$ are convergent; this implies $\lim\limits_{m\to \infty}\mathbb{E}\big[\alpha_m^{\prime}(K^{\prime}RK+Q_\mu)\alpha_m\big]=0$ and $\lim\limits_{m\to \infty}\mathbb{E}\big[\varphi_m^{\prime}(N^{\prime}m^{-1^{\prime}}RM^{-1}N+Q)\varphi_m\big]=0$. Then, we get $\lim\limits_{m\to \infty}\mathbb{E}[\alpha_m^{\prime}\alpha_m]=0$ and $\lim\limits_{m^{\prime}\to \infty}\mathbb{E}[\varphi_m^{\prime}\varphi_m]=0$. Hence, (\ref{new-system-form}) is bounded in the mean-square sense.
	
	Now we shall show that (\ref{inf_U}) and (\ref{inf_u_L}) are the optimal controllers. Define
	\begin{align*}
		\tilde{V}_k&=\mathbb{E}\big[x_k^{\prime}Z\hat{x}_{k\mid k}^R+x_k^{\prime}X_k(\hat{x}_{k\mid k}^L-\hat{x}_{k\mid k}^R)+x_k^{\prime}S\mathbb{E}x_k\big]\\
		&\quad +\mathbb{E}\sum_{i=k}^{\infty}\Big\{(x_i-\hat{x}_{i\mid i}^L)^{\prime}[(1-p)A^{\prime}ZWCA+pA^{\prime}C^{\prime}WXWCA+Q_\omega](x_i-\hat{x}_{i\mid i}^L)\Big\}\\
		&\quad+\sum_{i=k}^{\infty}tr\big\{Q_\omega[(1-p)Z_{k+1}W_{k+1}C+pC^{\prime}W_{k+1}C]+pQ_vW_{k+1}X_{k+1}W_{k+1}\big\}.
	\end{align*}
	Then,
	\begin{align*}
		\tilde{V}_k-\tilde{V}_{k+1}=&\mathbb{E}\Big\{x_k^{\prime}Q_\mu x_k-\mu x_k^{\prime}Q\mathbb{E}x_k+U_k^{\prime}RU_k+(\tilde{u}_k^{L})^{\prime}R^L\tilde{u}_k^L\\
		&-\big(U_k-\mathbb{E}U_k+K_k(x_k-\mathbb{E}x_k)\big)^{\prime}\Upsilon_k\big(U_k-\mathbb{E}U_k+K_k(x_k-\mathbb{E}x_k)\big)\\
		&-\big(\mathbb{E}U_k+(K_k+\bar{K}_k)\mathbb{E}x_k\big)^{\prime}M_k\big(\mathbb{E}U_k+(K_k+\bar{K}_k)\mathbb{E}x_k\big)\\
		&-\big(\tilde{u}_k^L+\Lambda_k^{-1}L_k(\hat{x}_{k\mid k}^L-\hat{x}_{k\mid k}^R)\big)^{\prime}\Lambda_k\big(\tilde{u}_k^L+\Lambda_k^{-1}L_k(\hat{x}_{k\mid k}^L-\hat{x}_{k\mid k}^R)\big)\Big\}.
	\end{align*}
	Taking summation from $k=0$ to $k=N$ on the both sides of the above equation, the cost function (\ref{inf_J}) becomes
	\begin{align*}
		\bar{J}_{\infty}(\mu)=&\mathbb{E}\Big\{\big(U_k-\mathbb{E}U_k+K_k(x_k-\mathbb{E}x_k)\big)^{\prime}\Upsilon_k\big(U_k+\mathbb{E}U_k+K_k(x_k-\mathbb{E}x_k)\big)\\
		&+\big(\mathbb{E}U_k+(K_k+\bar{K}_k)\mathbb{E}x_k\big)^{\prime}M_k\big(\mathbb{E}U_k+(K_k+\bar{K}_k)\mathbb{E}x_k\big)\\
		&+\big(\tilde{u}_k^L+\Lambda_k^{-1}L_k(\hat{x}_{k\mid k}^L-\hat{x}_{k\mid k}^R)\big)^{\prime}\Lambda_k\big(\tilde{u}_k^L+\Lambda_k^{-1}L_k(\hat{x}_{k\mid k}^L-\hat{x}_{k\mid k}^R)\big)\Big\}\\
		&+tr\big\{\Sigma^L[(1-p)A^{\prime}ZWCA+pA^{\prime}C^{\prime}WXWCA+Q_\omega]\big\}\\
		&+tr\{Q_\omega[(1-p)Z_{k+1}W_{k+1}C+pC^{\prime}W_{k+1}C]+pQ_vW_{k+1}X_{k+1}W_{k+1}\}.
	\end{align*}
	Due to $M>0$, $\Upsilon>0$ and $\Lambda>0$, the optimal controller can be obtained as $(\ref{inf_U})$ and $(\ref{inf_u_L})$.
	
	\textbf{"Necessity"}:
	Suppose that (\ref{new-system-form}) is bounded in the mean-square sense, we will show that there exist solutions $Z$, $X$, $S$ to (\ref{inf_Z}), (\ref{inf_X}) and (\ref{inf_S}) such that $Z>0$, $Z+S>0$ and $\Theta>0$ and $\sqrt{p}\left|\lambda_{\max }\left(A-B^{L} \Lambda^{-1} L\right)\right|<1$.
	From Theorem 2 of \cite{Xiao20}, we know: the fact that system (\ref{new-system-form}) is bounded in the mean-square sense is equivalent to that system (\ref{inf_system}) is stabilizable in the mean-square sense. From Lemma \ref{lemma4}, if the system (\ref{inf_system}) is stabilizable in the mean-square sense, then Riccati equations (\ref{inf_Z}), (\ref{inf_X}) and (\ref{inf_S}) admit solution $Z$, $X$ and $S$ satisfying $Z>0$, $Z+S>0$ and $\Theta>0$. Therefore, if system (\ref{new-system-form}) is bounded in the mean-square sense, the same conclusion can be obtained.
	
	Now we shall show that $\sqrt{p}\left|\lambda_{\max }\left(A-B^L\Lambda^{-1}L\right)\right|<1$. Clearly, if system (\ref{new-system-form}) is bounded in the mean-square sense, it follows that $\lim \limits_{k \rightarrow \infty} \mathbb{E}\left[x_{k}^{\prime} x_{k}\right]$ is bounded. Then we have
	\[\begin{array}{l}
		\lim \limits_{k \rightarrow \infty} \mathbb{E}\left[x_{k}^{\prime} x_{k}\right] 
		=\lim \limits_{k \rightarrow \infty} \mathbb{E}\Big\{\big[(x_{k}-\hat{x}_{k \mid k}^{R})+\hat{x}_{k \mid k}^{R}\big]^{\prime}\big[(x_{k}-\hat{x}_{k \mid k}^{R})+\hat{x}_{k \mid k}^{R}\big]\Big\} \\
		\hphantom{\lim \limits_{k \rightarrow \infty} \mathbb{E}\left[x_{k}^{\prime} x_{k}\right] }=\lim \limits_{k \rightarrow \infty} {tr}\Sigma_{k \mid k}^{R}+\lim \limits_{k \rightarrow \infty} \mathbb{E}[\hat{x}_{k \mid k}^{R^{\prime}} \hat{x}_{k \mid k}^{R}].
	\end{array}\]
	Hence, if $\lim \limits_{k \rightarrow \infty} tr\Sigma_{k \mid k}^{R}$ exists, then $\lim \limits_{k \rightarrow \infty} \Sigma_{k \mid k}^{R}$ is convergent. Noting Lemma \ref{lemma5} and under Assumption \ref{ass3} and Assumption \ref{ass4}, $\Sigma_{k \mid k}^{R}$ is asymptotic bounded if and only if $\sqrt{p}\left|\lambda_{\max }\left(A-B^L\Lambda^{-1}L\right)\right|<1$. This completes the proof of the necessity.
\end{proof}

Up to the present, we have found the optimal local and remote controllers when
$\mu$ is fixed for the finite-horizon case and infinite-horizon case, respectively. Next, in the next section, we shall show how to get the optimal multiplier $\mu^{*}$.

\section{Recovery of Primal-Optimal Solutions}

For any fixed $\mu\geq 0$, note the optimal control  $u^{\ast}(\mu)=(u_{0:N}^{L*},u_{0:N}^{R*})$ with $u_{k}^{L*},u_{k}^{R*}$ defined in (\ref{optimal_u_k^R}) and (\ref{optimal_u_k^L}). The following result is about the optimal Lagrangian multiplier $\mu^{\ast}$, which can be derived by standard optimization theory. 

\begin{theorem}\label{re}
	
Define the multiplier
	\begin{equation}
		\mu^{\ast}\triangleq \inf\{\mu \geq 0:J_R(u^{\ast}(\mu)\leq{\epsilon}\}.
	\end{equation}
	If $\mu^{\ast}$ is finite, then the policy $u^{\ast}(\mu^{\ast})$ is optimal for the primal Problem (CLQ).
\end{theorem}
\begin{proof}
	The proof is similar to that of [\cite{Tsiamis20},Theorem 3]. Thus we omit here.
\end{proof}
Theorem \ref{re} implies that we can find an optimal Lagrangian multiplier $\mu^{*}$ by performing the simple bisection on $\mu$  (\cite{Tsiamis20}).
In the process of finding $\mu^{*}$, we can get the value
of $J_R(u^{*}(\mu))$ through the following theorem.
\begin{theorem}
	For fixed $\mu\geq 0$,  $J_R(u^{*}(\mu))$ is expressed as
	\begin{align}
		J_R(u^{*}(\mu))=\mathbb{E}\Big\{x_0^{\prime}O_0\hat{x}_{0\mid 0}^R+x_0^{\prime}P_0(\hat{x}_{k\mid k}^L-\hat{x}_{0\mid 0}^R)+x_0^{\prime}W_0\mathbb{E}x_0\Big\}+q_0,
	\end{align}
where $O_k$, $P_k$, $W_k$ and $q_k$, $k=0,\cdots N+1$
are evaluated through the backward recursions
\begin{align*}
	O_k&=Q+A^{\prime}O_{k+1}A-2K_k^{\prime}B^{\prime}O_{k+1}A+K_k^{\prime}B^{\prime}O_{k+1}B,\\
	P_k&=(1-p)A^{\prime}O_{k+1}A-2(1-p)L_k^{\prime}\Lambda_k^{-1^{\prime}}B^LO_{k+1}A+(1-p)L_k^{\prime}\Lambda_k^{-1^{\prime}}B^{L^{\prime}}O_{k+1}\Lambda_k^{-1}L_k\\
	&\quad +pA^{\prime}P_{k+1}A-2pA^{\prime}P_{k+1}B^L\Lambda_k^{-1}L_k+pA^{\prime}O_{k+1}A\\
	&\quad+(1-p)L_k^{\prime}\Lambda_k^{-1^{\prime}}B^{L^{\prime}}P_{k+1}\Lambda_k^{-1}L_k+Q,\\
	W_k&=-Q-\bar{K}_k^{\prime}B^{\prime}O_{k+1}A+2K_k^{\prime}B^{\prime}O_{k+1}B\bar{K}_k+\bar{K}_k^{\prime}B^{\prime}O_{k+1}B\bar{K}_k+A^{\prime}W_{k+1}A\\
	&\quad-K_k^{\prime}B^{\prime}W_{k+1}A-2\bar{K}_kB^{\prime}W_{k+1}A+K_k^{\prime}B^{\prime}W_{k+1}BM_{k}^{-1}N_k+\bar{K}_k^{\prime}BW_{k+1}BM_{k}^{-1}N_k,\\
	q_k&=tr(\Sigma_{k\mid k}^L(Q+A^{\prime}O_{k+1}A))+q_{k+1}
\end{align*}
with $O_{N+1}=P_{N+1}=Q$, $W_{N+1}=-Q$ and $q_{N+1}={tr}(\Sigma_{N+1\mid N+1}^LQ)$.
\end{theorem}
\begin{proof}
	The method proved here is similar to that in Theorem
	\ref{theorem2}; so it is omitted.
\end{proof}

So far, the closed-form solution of the considered optimal  risk-constrained controllers are obtained that are parameterized by $\mu^{*}$.  
According to the zero duality gap and $(\ref{13}) (\ref{optimal_J})$, the optimal cost can be calculated.

\section{Numerical Examples}
In this section, we give two numerical examples to verify the effectiveness of the proposed results.

 \textbf{Example 1}. Consider a dynamic system of form (\ref{system}) with one remote controller and one local controller whose parameters are given by  
\begin{align*}
	&\ A=\begin{bmatrix}
		4 & 1\\
		1 & 0.1
	\end{bmatrix},\quad
	B=\begin{bmatrix}
		1&1&1&0\\
		0&1&1&1
	\end{bmatrix},\quad
	B^L=\begin{bmatrix}
		1&1\\
		0&1
	\end{bmatrix},&\\
	&\ C=\begin{bmatrix}
		1&1\\
		0&-1
	\end{bmatrix},\quad
	Q=\begin{bmatrix}
		1&0\\0&1
	\end{bmatrix},\quad
		G=\begin{bmatrix}
		1&0\\0&1
	\end{bmatrix}, &\\
	&\ R=\begin{bmatrix}
		1&0&0&0\\
	0&1&0&0\\
	0&0&1&0\\
	0&0&0&1
	\end{bmatrix},\quad
	R^L=\begin{bmatrix}
		1&0\\
		0&1
	\end{bmatrix},\quad
	\Sigma_{init}=\begin{bmatrix}
		1&0\\
		0&1
	\end{bmatrix}.
\end{align*}
Let system noise $\omega_k$, observation noise $v_k$ and  initial state $x_0$ obey the Gaussian distribution:
	\begin{align*}
	\omega_k\sim \mathcal{N}\left(\begin{bmatrix}
		0\\0
	\end{bmatrix},\begin{bmatrix}
		10&0\\0&10
	\end{bmatrix}\right),~v_k\sim \mathcal{N}\left(\begin{bmatrix}
	0\\0
	\end{bmatrix},\begin{bmatrix}
	10&0\\0&10
	\end{bmatrix}\right),~x_0\sim \mathcal{N}\left(\begin{bmatrix}
	1\\1
	\end{bmatrix},\begin{bmatrix}
	1&0\\0&1
	\end{bmatrix}\right),~k=0,\cdots,50,
	\end{align*}
which are i.i.d.. 
Test Problem (FLQ) with $\mu=10$ and $\mu=0$ (the constraint-free case).


We run the MATLAB code and obtain 1000 sample trajectories of state, respective, for $\mu=10$ and $\mu=0$ with link failure probability $p=0.5$. In Figure 2 below, the trajectories of $\mathbb{E}\big(x_k-\mathbb{E}x_k\big)^{\prime}Q\big(x_k-\mathbb{E}x_k\big)$ are presented, where the red curve is the trajectory for $\mu=10$ and the green one is for $\mu=0$.  
It is clear that the state trajectory of constraint-free case has larger variability; in other words, by posing the constraint (\ref{constraint-0}) the obtained optimal state trajectory becomes flatter.


\begin{figure}[!htb] 
	\centering	\includegraphics[width=0.6\hsize]{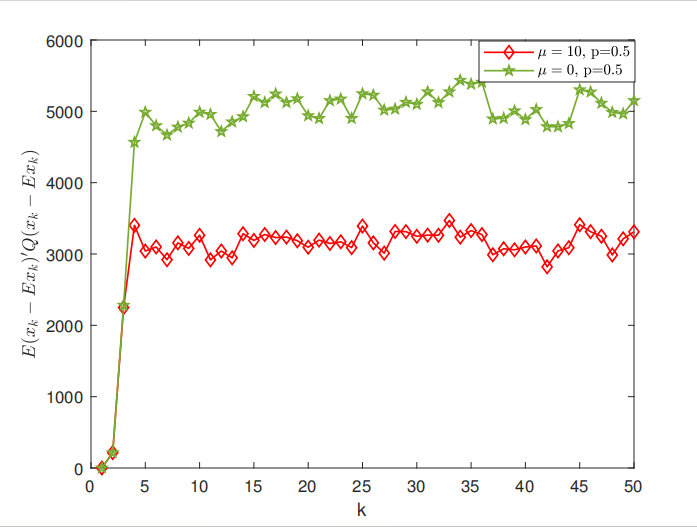}
	\caption{Evolution of $\mathbb{E}\big(x_k-\mathbb{E}x_k\big)^{\prime}Q\big(x_k-\mathbb{E}x_k\big)$ }
	\label{fig2}
\end{figure}

Then, we plot the trace trajectory of error covariance matrix $\Sigma_{k\mid k}^R$ of (\ref{es_x_R}) when the link failure probability takes different values.
As can be seen from Figure 3, the traces of error covariance of $p=0.2$ are smaller those of $p=0.8$.
\begin{figure}[!htb]
	\centering
	\includegraphics[width=0.6\hsize]{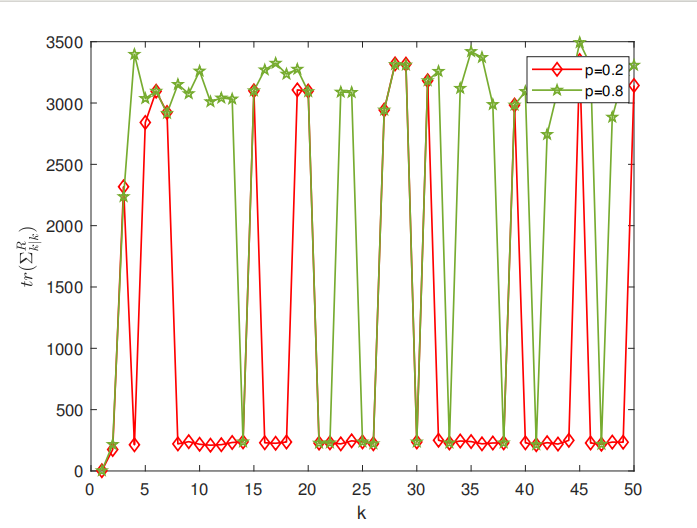}
	\caption{Traces of error covariance matrices}.
	\label{fig3}
\end{figure}

\textbf{Example 2}. Given the level $\epsilon$ of $(\ref{constraint-0})$, find the optimal multiplier $\mu^{*}$ by performing simple bisection on $\mu$.  
%
The parameters are given below
\begin{align*}
	&\ A=\begin{bmatrix}
		2 & 0.1\\
		1 & 0.1
	\end{bmatrix},\quad
	B=\begin{bmatrix}
		1&1&1&0\\
		0&1&1&1
	\end{bmatrix},\quad
	B^L=\begin{bmatrix}
		1&1\\
		0&1
	\end{bmatrix},&\\
	&\ C=\begin{bmatrix}
		1&1\\
		0&-1
	\end{bmatrix},\quad
	Q=\begin{bmatrix}
		1&0\\0&1
	\end{bmatrix},\quad
	G=\begin{bmatrix}
		1&0\\0&1
	\end{bmatrix},&\\
	&\ R=\begin{bmatrix}
		1&0&0&0\\
		0&1&0&0\\
		0&0&1&0\\
		0&0&0&1
	\end{bmatrix},\quad
	R^L=\begin{bmatrix}
		1&0\\
		0&1
	\end{bmatrix},\quad
	\Sigma_{init}=\begin{bmatrix}
		1&0\\
		0&1
	\end{bmatrix},~~ p=0.5.
\end{align*}
System noise $\omega_k$, observation noise $v_k$ and initial state $x_0$ obey the Gaussian distribution:
\begin{align*}
	\omega_k\sim \mathcal{N}(\begin{bmatrix}
		0\\0
	\end{bmatrix},\begin{bmatrix}
		1&0\\0&1
	\end{bmatrix}),~v_k\sim \mathcal{N}(\begin{bmatrix}
		0\\0
	\end{bmatrix},\begin{bmatrix}
		1&0\\0&1
	\end{bmatrix}),~x_0\sim \mathcal{N}(\begin{bmatrix}
		0\\0
	\end{bmatrix},\begin{bmatrix}
		1&0\\0&1
	\end{bmatrix}),~k=0,\cdots,5,
\end{align*}
which are i.i.d..
We set $\epsilon=40$. By using the method of bisection, we obtain that $\mu^{*}=6.25$, 
 and the function $(\ref{constraint-0})$ $J_R=38.65$.

\section{Conclusion}
To reduce the oscillation of system state, a risk constraint is posed on the cumulative state weighted variance of a partially-observed decentralized stochastic LQ problem with one remote controller and one local controller.
By punishing the risk constraint into the cost function through the Lagrange multiplier method, the resulting augmented cost function will include a quadratic mean-field term of state.
For fixed Lagrange multiplier $\mu$, explicit expressions of the optimal control strategy are obtained for the corresponding  finite-horizon and infinite-horizon optimal control problems together with the necessary and sufficient conditions for the system to be mean-square bounded.
Furthermore, by using  the bisection method, optimal Lagrange multiplier is computed. 
%

\bibliographystyle{plainnat}
\bibliography{JiaHui}

\end{document}